\documentclass{amsart}
\usepackage{amssymb,amsmath,amsthm,graphics}
\usepackage{texdraw}
\newtheorem{theorem}{Theorem}
\newtheorem{definition}{Definition}
\newtheorem{proposition}{Proposition}
\newtheorem{lemma}{Lemma}
\newtheorem{corollary}{Corollary}

\newtheorem{rmk}{Remark}
\newenvironment{remark}{\begin{rmk}\rm}{\end{rmk}}

\renewcommand{\ne}{\not =}
\newcommand{\obra}[3]{{#1}: {\emph{#2}.\/}  {#3}.}
\title{Reduction of Singularities of Three-Dimensional Line Foliations}
\author{F. Cano; C. Roche; M. Spivakovsky }
\date{2010 april, the 30th}
\begin{document}
\maketitle
\begin{flushright}
\em Dedicated to Heisuke Hironaka on the occasion of his 80th
birthday.
\end{flushright}

\setcounter{section}{-1} \tableofcontents
\section{Introduction}
We give a birational reduction of singularities for one dimensional
foliations in ambient spaces of dimension three. To do this, we
first prove the existence of a Local Uniformization in the sense of
Zariski \cite{Zar}. The reduction of singularities is then obtained
by a gluing procedure for Local Uniformization similar to Zariski's
one in {\cite{Zar2}}.

Let $K$ be the field of rational functions of a projective algebraic
variety $M_0$ of dimension $n$ over an algebraically closed field
$k$ of characteristic zero. We prove the following theorem

\begin{theorem}[Local Uniformization] \label{teouno} Assume that $n=3$. Consider  a
$k$-valuation $\nu$ of $K$ and a foliation by lines ${\mathcal
L}\subset \mbox{Der}_kK$. There is a composition of a finite
sequence of blow-ups with non singular centers $M\rightarrow M_0$
such that $\mathcal L$ is log-elementary at the center $Y\subset M$
of $\nu$.
\end{theorem}

 A {\em foliation by lines\/} (or simply a {\em foliation\/})  is any $1$-dimensional
$K$-vector subspace ${\mathcal L}\subset\mbox{Der}_kK$. Recall that
 space of $k$-derivations $ \mbox{Der}_kK $ is a $n$-dimensional $K$-vector space.
The notion of ``log-elementary'' comes from results in \cite{Can}.
Let us explain it. Take a regular point $P$ in a projective model
$M$. We know that $\mbox{Der}_k{\mathcal O}_{M,P}\subset
\mbox{Der}_kK$ is a free ${\mathcal O}_{M,P}$-module of rank $n$
generated by the partial derivatives $\partial/\partial x_i$,
$i=1,2,\ldots,n$, for any regular system of parameters
  $x_1,x_2,\ldots,x_n$  of the
local ring ${\mathcal O}_{M,P}$.
 Moreover $${\mathcal L}_{M,P}={\mathcal
L}\cap \mbox{Der}_k{\mathcal O}_{M,P}$$ is a free rank one
sub-module of $\mbox{Der}_k{\mathcal O}_{M,P}$ that we call the {\em
local foliation induced by } $\mathcal L$ at $M,P$. We say that
${\mathcal L}$ is {\em non-singular\/} at $P$ if  $ {\mathcal
L}_{M,P}\not\subset{\mathcal M}_{M,P}\mbox{Der}_k{\mathcal
O}_{M,P}$, where ${\mathcal M}_{M,P}\subset{\mathcal O}_{M,P}$ is
the maximal ideal. We say that ${\mathcal L}$ is {\em log-elementary
\/} at $P$ if there is a regular system of parameters
$z_1,z_2,\ldots,z_n$, an integer $0\leq e\leq n$ and $\xi\in
{\mathcal L}_{M,P}$ of the form
$$
\xi=\sum_{i=1}^e a_iz_i\frac{\partial}{\partial z_i}+\sum_{i=e+1}^n
a_i\frac{\partial}{\partial z_i},\; (a_i\in{\mathcal O}_{M,P},
i=1,2,\ldots,n)
$$
with $a_j\notin{\mathcal M}^2_{M,P}$ for at least one index $j$. If
$Y\subset M$ is an irreducible subvariety, we say that $\mathcal L$
is {\em non-singular at $Y$,} respectively, {\em log-elementary at
$Y$,} if it is so at a generic point of $Y$. Note in particular that
$M$ must be non-singular at a generic point of $Y$.

Theorem \ref{teouno} may be globalized as a consequence of a
patching procedure developed by O. Piltant \cite{Pil}, which is an
axiomatic adaptation of the one given by Zariski in the case of
varieties \cite{Zar}. We obtain the following birational  result of
reduction of singularities of foliations in an ambient space of
dimension three
\begin{theorem}
\label{teodos} Assume that $n=3$ and let ${\mathcal L}\subset
\mbox{Der}_kK$ be a foliation. Consider a birational model $M_0$ of
$K$. There is a birational morphism $M\rightarrow M_0$ such that
$\mathcal L$ is log-elementary at all the points of $M$.
\end{theorem}

The reduction of singularities of foliations in an ambient space of
dimension two is proved in the classical Seidenberg's paper
\cite{Sei}. In dimension three or higher one would like to be able
to obtain {\em elementary singularities}, that is singularities with
a non-nilpotent linear part. This is not possible in a birational
way as an example of F. Sanz and F. Sancho shows (see for instance
the introduction of \cite{Pan}). There is no general result in
dimension $n\geq 4$, except for the case of absolutely isolated
singularities \cite{Cam-C-S}. In dimension three Panazzolo
\cite{Pan} gives a global but non-birational result over the real
numbers, getting elementary singularities after doing ramifications
and blow-ups. There is also a preprint of Panazzolo and McQuillan,
where they announce and adaptation to the results in \cite{Pan} to
the language of stacks. In \cite{Can-M-R} there is a local result,
along a trajectory of a real vector field, obtained also by the use
of ramifications and blow-ups. Finally, in \cite{Can} there is a
strategy to solve by means of blow-ups a ``formal version'' of the
local uniformization problem, where formal non-algebraic centers of
blow-up are allowed.

Let us give an outline of the proof of Theorem \ref{teouno}. We
organize the proof by taking account of the ranks and dimension of
the valuation and of the existence of ``maximal contact'' with a
formal series.

In Part I, we consider the case of a real valuation $
\nu:K\setminus\{0\}\rightarrow {\mathbb R}$ with residual field
$\kappa_\nu=k$. In the classical situations of Zariski's Local
Uniformization \cite{Zar} this one is considered to be the most
difficult case.  Note that since $\kappa_\nu=k$ the center of $\nu$
at any projective model is a closed point. Our first result is
\begin{theorem}
\label{teo:tres} Assume that $n=3$ and  $\nu$ is a real
$k$-valuation of $K$ with resi\-dual field $\kappa_\nu=k$. There is
a finite composition of blow-ups with non-singular centers
$M\rightarrow M_0$ such that $M$ is non-singular at the center $P$
of $\nu$ at $M$ and and one of the following properties holds
\begin{enumerate}\item
 $\mathcal L$ is log-elementary at $P$.
 \item
There is $\hat f\in \widehat {\mathcal O}_{M,P}$
 having transversal maximal contact with $\nu$.
 \end{enumerate}
\end{theorem}
A formal series $\hat f\in \widehat {\mathcal O}_{M,P}$
 has {\em transversal maximal contact with $\nu$} if it is
 the Krull-limit of a sequence
  $f_i\in {\mathcal O}_{M,P}$ with strictly
increasing values and moreover we have the following property of
transversality: there is a part of a regular system of parameters
$x_1,x_2,\ldots,x_r$ of ${\mathcal O}_{M,P}$ such that the values
$\nu(x_1),\nu(x_2),\ldots,\nu(x_r)$ are ${\mathbb Z}$-independent,
where $r$ is the rational rank of $\nu$, and
$x_1,x_2,\ldots,x_r,\hat f$ is a part of a regular system of
parameters of the complete local ring $\widehat{\mathcal O}_{M,P}$.

In order to prove Theorem \ref{teo:tres}, we work over the rational
rank $r$ of $\nu$ and we study the three following cases in an
ordered way:\begin{enumerate}
\item $r=n$. Here we get $\mathcal L$   elementary
for any ambient dimension $n$. This is a { combinatorial case} with
few differences with respect to the classical situations of
varieties.

\item $r=n-1$. The statement of  Theorem \ref{teo:tres}
is valid for any $n$. We use Newton Polygon technics to give the
proof. If $n=2$ the result is slightly stronger: we get either
maximal contact of a non-singular foliation. This will be useful in
the next case.
\item $r=1, n=3$. This is the hardest situation. We have
important difficulties due to the fact that $\nu$ is not a discrete
valuation.
\end{enumerate}
We end Part I by giving a proof of
\begin{theorem}
\label{teo:cuatro} Assume that $n=3$. Let $\nu$ be a real
$k$-valuation of $K$ with residual field $\kappa_\nu=k$ and suppose
that $\hat f\in \widehat {\mathcal O}_{M,P}$ has transversal maximal
contact with $\nu$. There is a finite composition of blow-ups with
non-singular centers $M\rightarrow M_0$ such that $\mathcal L$ is
log-elementary at the center $P$ of $\nu$.
\end{theorem}
Part II is devoted to the remaining cases. We obtain many of the
results by an inductive use of the technics in Part I. In Part III
we prove the validity of Piltant's patching axioms and hence we
obtain the proof of Theorem \ref{teodos}.

\part{Zero dimensional arquimedean valuations}
In all this part $\nu:K\setminus\{0\}\rightarrow \Gamma$ denotes a
valuation such that $\Gamma\subset({\mathbb R},+)$ and
$\kappa_\nu=k$. In other words, the (arquimedean) rank of $\nu$ is
one and it is a zero-dimensional $k$-valuation of $K$. We denote by
$r$ the {\em rational rank\/} of $\nu$, that is, the maximum number
of $\mathbb Z$-linearly independent elements in the value group of
$\nu$. We know that $1\leq r \leq n$ by Abhyankar's inequality. In
particular, for the case $n=3$ we have the possibilities $r=3,r=2$
and $r=1$.
\section{Parameterized regular local models}
\label{seccion:uno} A {\em para\-meterized regular local model\/}
${\mathcal A}=({\mathcal O},{\mathbf z}=({\mathbf x},{\mathbf y}))$
for $K,\nu$ is a pair with ${\mathcal O}={\mathcal O}_{M,P}$, where
$M$ is a projective model of $K$, the point $P\in M$ is the center
of $\nu$ in $M$ and the sequence
$$(z_1,z_2,\ldots,z_n)={\mathbf z}=({\mathbf x},{\mathbf y})=(x_1,x_2,\ldots,x_r,y_{r+1},y_{r+2},\ldots,y_n)$$ is a regular
system of parameters of $\mathcal O$ such that
$\nu(x_1),\nu(x_2),\ldots,\nu(x_r)$ are $\mathbb Z$-linearly
independent values. We call ${\mathbf x}=(x_1,x_2,\ldots,x_r)$ the
{\em independent variables} and ${\mathbf
y}=(y_{r+1},y_{r+2},\ldots,y_n)$ the {\em dependent variables}. The
existence of parameterized regular local models is a consequence of
Hironaka's reduction of singularities \cite{Hir}. More precisely, we
have
\begin{proposition} \label{prop:uno} Given a projective model $M_0$
of $K$, there is a composition of a finite sequence of blow-ups with
non-singular centers $M\rightarrow M_0$ such that the center $P$ of
$\nu$ at $M$ provides a local ring ${\mathcal O}={\mathcal O}_{M,P}$
for a parameterized regular local model ${\mathcal A}=({\mathcal
O},{\mathbf z}=({\mathbf x},{\mathbf y}))$.
\end{proposition}
\begin{proof}
 By Hironaka's reduction of the singularities (see \cite{Hir}) of $M_0$, we get a
nonsingular projective model $M'$ of $K$ jointly with a birational
morphism $M'\rightarrow M_0$ that is the composition of a finite
sequence of blow-ups with non-singular centers. Consider the local
ring ${\mathcal O}_{M',P'}$ of $M'$ at the center $P'$ of $\nu$ and
chose elements $f_1,f_2,\ldots,f_r\in{\mathcal O}_{M,P}$ such that
$\nu(f_1),\nu(f_2),\ldots,\nu(f_r)$ are $\mathbb Z$-linearly
independent.  Another application of Hironaka's theorem gives a
birational morphism $M\rightarrow M'$, that is also a  composition
of a finite sequence of blow-ups with non-singular centers, such
that$f=\prod_{i=1}^rf_i$,  is a monomial (times a unit) in a
suitable regular system of parameters at any point of $M$ and hence
each of the $f_i$, $i=1,2,\ldots,r$ is also a monomial (times a
unit) in that regular system of parameters. In particular, if $P$ is
the center of $\nu$ at $M$ there is a regular system of parameters
${\mathbf z}=(z_1,z_2,\ldots,z_n)$ of ${\mathcal O}_{M,P}$ such that
$$ f_i=U_i{\mathbf z}^{{\mathbf m}_i}, U_i\in {\mathcal O}_{M,P}\setminus{\mathcal
M}_{M,P},\mbox{ for  } i=1,2,\ldots, n,
$$ where
${\mathbf m}_i=(m_{i,1}, m_{i,2}, \ldots, m_{i,n})\in{\mathbb
Z}_{\geq 0}^n$ and ${\mathbf z}^{{\mathbf
m}_i}=z_1^{m_{i,1}}z_2^{m_{i,2}}\cdots z_n^{m_{i,n}}$. In terms of
values, we have $ \nu(f_i)=\sum_{j=1}^nm_{ij}\nu(z_j)$. This implies
that there are $r$ variables among the  $z_j$ whose values are
$\mathbb Z$-linearly independent.
\end{proof}
\subsection{Coordinate changes and blow-ups}
Take a  parameterized regular local model ${\mathcal A}=({\mathcal
O},{\mathbf z})$. We will do ``atomic'' transformations of
${\mathcal A}$ of two types: {\em coordinate changes in the
dependent variables} and {\em coordinate blow-ups with codimension
two centers}. Our ``basic'' transformations, called {\em Puiseux
packages\/} will be
 certain sequences of coordinate changes and blow-ups.

  Let us describe the
 two types of transformations. Each one produces a parameterized
 local model ${\mathcal A}'=({\mathcal O}',{\mathbf z}')$.
 \par
 {\em Coordinate changes in the dependent variables.\/} Consider $j$ with
 $r+1\leq j\leq n$. A {\em $j$-coordinate change} is
  such that  $z'_i=z_i$ for $i\ne j$ and
$y'_j$ is  one of the following
\begin{enumerate}\item[a)]
$ y'_j=y_j-c{\mathbf x}^{\mathbf a}, \; \nu(y'_j)\geq \nu(y_j),\,
c\in k,\, {\mathbf a}\in{\mathbb Z}_{\geq 0}^r$.
\item[b)] $y'_j=y_j+y_s$,
 for another $s\ne j$ with $r+1\leq s\leq n$.
\end{enumerate}
If $r=n$ we do not do coordinate changes.

 {\em Coordinate
blow-ups with codimension two centers.\/} Take a pair $i,j$ of
distinct indices with $1\leq i\leq r$ and $1\leq j\leq n$. We say
that ${\mathcal A}'=({\mathcal O},{\mathbf z}')$ is obtained from
${\mathcal A}$ by an $(i,j)$-{\em blow-up} if the following holds.
First $z'_s=z_s$ for any $s\notin\{i,j\}$. In order two determine
$z'_i,z'_j$ we have three cases
\begin{enumerate}
\item $\nu(x_i)<\nu(z_j)$. We put $x'_i=x_i$ and $z'_j=z_j/x_i$.
\item $\nu(x_i)>\nu(z_j)$. We put $x'_i=x_i/z_j$ and $z'_j=z_j$.
\item $\nu(x_i)=\nu(z_j)$. Note that in this case we necessarily
have that $j\geq r+1$ and hence $z_j=y_j$. Since $\kappa_\nu=k$,
there is $c\in k$ with $\nu(y_j/x_i-c)>0$. We put $x'_i=x_i$ and
$y'_j=y_j/x_i-c$.
\end{enumerate}
The first two cases above are called {\em combinatorial\/} and the
third one corresponds to a blow-up {\em with translation}. If
$x_i,x_j$ are independent variables, we have always a combinatorial
case, since $\nu(x_i)\ne \nu(x_j)$.

The local ring ${\mathcal O}'$ is the (algebraic) localization of
${\mathcal O}[{\mathbf z}']$ at the ideal $({\mathbf z}')$.

In the case that $j\geq r+1$ the above blow-up will also be referred
as a $j$-{\em blow-up}.
\begin{remark}
\label{observacionuno} Let  $M$ be a projective model for $K$ such
that ${\mathcal O}={\mathcal O}_{M,P}$, where $P$ is the center of
$\nu$ at $M$. There is a closed irreducible algebraic subvariety
$Y\subset M$ of codimension two defined by the equations $x_i=z_j=0$
that is non singular at $P$. Let $\pi:M'\rightarrow M$ be the
blow-up of $M$ with center $Y$ and let $P'$ be the center of $\nu$
at $M'$. Then ${\mathcal O}'={\mathcal O}_{M',P'}$.
\end{remark}

\subsection{Puiseux packages of blow-ups} Let ${\mathcal
A}=({\mathcal O},{\mathbf z}=({\mathbf x},{\mathbf y}))$ be a
para\-meterized regular local model. Consider a dependent variable
$y_j$. Then $\nu(y_j)$ can be expressed uniquely as a $\mathbb
Q$-linear combination of  $\nu(x_1),\nu(x_2),\ldots,\nu(x_r)$. More
precisely, there are unique integer numbers $d>0$ and
$p_1,p_2,\ldots,p_r$ such that
$$
d\nu(y_j)=p_1\nu(x_1)+p_2\nu(x_2)+\cdots+p_r\nu(x_r)
$$
and $gcd(d;p_1,p_2,\ldots,p_r)=1$. In particular, the rational
function
$$
\Phi=y_j^d/{\mathbf x}^{\mathbf p} ,\quad {\mathbf x}^{\mathbf
p}=x_1^{p_1}x_2^{p_2}\cdots x_r^{p_r},
$$
has value equal to zero. We call this function {\em the $j$-contact
rational function\/} and $d$ is the {\em $j$-ramification index for
$\mathcal A$}. Note that there is a unique scalar $c\in k$ such that
$\nu(\Phi-c)>0$, since $\kappa_\nu=k$.

A coordinate $(i,s)$-blow-up is said to be  $j$-{\em admissible\/}
if either $1\leq s\leq r$ with $p_i\ne 0\ne p_s$ or $p_i\ne 0$ and
$s=j$.
\begin{remark}
Assume that ${\mathcal A}'$ has been obtained from ${\mathcal A}$ by
a  $j$-admissible coordinate $(i,s)$-blow-up. There are two
possibilities:
\begin{enumerate}
\item[A)] The blow-up is  combinatorial. In this case
$\Phi$ is also the $j$-contact rational function for ${\mathcal
A}'$.
\item[B)] The blow-up has a translation. Then $\Phi=y_j/x_i$ and $s=j$.
Moreover, we have $y'_j=\Phi-c$.
\end{enumerate}
\end{remark}
\begin{definition} A $j$-{\em Puiseux package\/} starting at ${\mathcal A}$ is a finite
sequence
$$
{\mathcal A}={\mathcal A}_0\rightarrow {\mathcal
A}_1\rightarrow\cdots\rightarrow {\mathcal A}_N={\mathcal A}'
$$
where  ${\mathcal A}_{t-1}\rightarrow {\mathcal A}_{t}$ is a
combinatorial $j$-admissible blow-up for $t=1,2,\ldots,N-1$ and
${\mathcal A}_{N-1}\rightarrow {\mathcal A}_{N}$ is a $j$-admissible
blow-up with translation. In this situation, we say that ${\mathcal
A}'$ has been obtained from ${\mathcal A}$ by a $j$-Puiseux package.
 \end{definition}
 Note that $y'_j=\Phi-c$, in view of the above Remark.
 \begin{proposition}
 \label{prop:dos}
  Given ${\mathcal A}$ and $j$,
 with $r<j\leq n$, there is at least one $j$-Puiseux package
 starting at ${\mathcal A}$.
 \end{proposition}
 \begin{proof}  There are many known algorithms for doing this
 (see \cite{Hir,Tei,Spi,Vil,Bie-M}). We include a proof for the sake
 of completeness. Let us write
 $$
\Phi=\frac{y_j^d{\mathbf x}^{\mathbf q}}{{\mathbf x}^{\mathbf r}},
 $$
 where  $q_i=-p_i$ if $p_i<0$ and $q_i=0$, otherwise and, in the same
 way, we put  $r_i=p_i$ if $p_i>0$ and $r_i=0$ otherwise.
 There are two possibilities: ${\mathbf q}\ne 0$ or ${\mathbf q}=0$.
 Note that we always have that ${\mathbf r}\ne 0$, since
 $\nu(z_s)>0$ for all $s$. Assume first that ${\mathbf q}\ne 0$. Let
 us choose indices $1\leq i,s\leq r$ such that $p_ip_s<0$. We do the
 $(i,s)$-blow-up. The sum $\vert p_i\vert +\vert p_s\vert $
 decreases. We continue and one of the independent variables $x_i$
 or  $x_s$ disappears. In this way we get that ${\mathbf q}=0$. Now,
 we consider an index $i$ with $p_i\ne 0$ and we do the
 $(i,j)$-blow-up. This blow-up is combinatorial except in the case that $\Phi=y_j/x_i$.
 If we are not in this case, then $d+p_i$ decreases and finally the variable $x_i$ disappears.
We obtain that $\Phi=y_j/x_i$. The only possible $j$-admissible
coordinate
 blow-up is the $(i,j)$-blow-up. Moreover, $\nu(y_j)=\nu(x_i)$ and
 hence it is a coordinate blow-up with translation.
 \end{proof}

\begin{remark}
\label{obsecuaciones} We are interested in the following features of
Puiseux packages. Let us start with ${\mathcal A}=({\mathcal
O},{\mathbf z}=({\mathbf x},{\mathbf y}))$ and assume that
${\mathcal A}'=({\mathcal O}',{\mathbf z}'=({\mathbf x}',{\mathbf
y}'))$ has been obtained from $\mathcal A$ by a $j$-Puiseux package.
Let  $\Phi=y_j^d/{\mathbf x}^{\mathbf p}$ be the $j$-contact
function and suppose that $\nu(\Phi-c)>0$. For $s\notin\{i;p_i\ne
0\}\cup\{j\}$ we have that $z_s=z'_s$. Moreover $y'_j=\Phi-c$ and
there are monomial expressions
$$
 z_s= \left(\prod_{i=1}^r{{ x}'_i}^{{
 b}^{s}_i}\right)\Phi^{{b}^{s}_j};\; s\in\{i;p_i\ne 0\}\cup\{j\}.
$$
  This is proved by
  induction on the number of
 $j$-admissible coordinate blow-ups of the $j$-Puiseux package.
\end{remark}

\subsection{Statements in terms of parameterized regular local
models} Consider a foliation by lines
 ${\mathcal L}\subset\mbox{Der}_kK$ and
 a parameterized regular local model ${\mathcal A}=({\mathcal
O}, {\mathbf z})$. The {\em local foliation induced by $\mathcal L$
at\/} $\mathcal A$ is defined by
$$
{\mathcal L}_{\mathcal A}={\mathcal L}\cap\mbox{Der}_k{\mathcal O}.
$$
Obviously ${\mathcal L}_{\mathcal A}={\mathcal L}_{M,P}$ for any
projective model $M$ for $K$ such that ${\mathcal O}={\mathcal
O}_{M,P}$. In the next sections we shall prove the following
proposition
\begin{proposition}\label{pro:tres}  Assume that $n=3$. Let $\nu$ be a real $k$-valuation of $K$ with
$\kappa_\nu=k$ and take a foliation ${\mathcal
L}\subset\mbox{Der}_kK$. Consider a parameterized regular local
model
 ${\mathcal A}=({\mathcal O}, {\mathbf z})$ for $K,\nu$.  There is a finite
sequence of coordinate changes and blow-ups such that the
parameterized regular local model ${\mathcal A}'=({\mathcal O}',
{\mathbf z}')$ obtained from $\mathcal A$ satisfies one of the
following properties:
\begin{enumerate}
\item The foliation ${\mathcal L}_{{\mathcal A}'}$
is log-elementary.
\item There is  $\hat f\in
\widehat {\mathcal O}'$ having transversal maximal contact with
$\nu$.
\end{enumerate}
\end{proposition}

This result implies Theorem \ref{teo:tres}. Indeed, we already know
that there is a birational morphism $M\rightarrow M_0$, composition
of blow-ups with nonsingular centers, such that $M$ is non-singular
and the local ring ${\mathcal O}_{M,P}$ of $M$ at the center $P$ of
$\nu$ supports a parameterized regular local model ${\mathcal A}$.
The sequence of blow-ups that gives ${\mathcal A}'$ may be
substituted, by Hironaka's reduction of singularities, by another
sequence of blow-ups with non-singular centers, since the original
blow-ups are non-singular (in fact they are non-singular and two
dimensional) at the corresponding centers of the valuation at each
projective model.

Next sections are devoted to proving Proposition \ref{pro:tres}.

\section{The combinatorial case $(r=n)$}
\label{secciondos} The following Proposition \ref{pro:cuatro}
implies Proposition \ref{pro:tres} for the case of maximal rational
rank. Let us note that in Proposition \ref{pro:cuatro} there is no
assumption about the (arquimedean) rank of the valuation nor on the
fact that $\kappa_\nu=k$. Indeed if $n=r$ we know that $\kappa_\nu$
is an algebraic extension of $k$ and thus $\kappa_\nu=k$ since we
assume the base field $k$ to be algebraically closed.
\begin{proposition}\label{pro:cuatro}
Let $\nu$ be a  $k$-valuation of $K$ with maximal rational rank
$r=n$. Take a foliation ${\mathcal L}\subset\mbox{Der}_kK$ and a
parameterized regular local model
 ${\mathcal A}$ for $K,\nu$.  There is a
 parameterized regular local model
 ${\mathcal A}'$ obtained from $\mathcal
 A$ by a finite sequence of coordinate blow-ups
such that ${\mathcal L}_{{\mathcal A}'}$ is elementary.
\end{proposition}

Let us recall that  ${\mathcal L}_{\mathcal A}$ is {\em elementary}
if there is a vector field $\xi\in{\mathcal L}_{\mathcal A}$ having
a non-nilpotent linear part. If $\xi\in\mbox{Der}_k{\mathcal O}$ is
singular, that is $\xi({\mathcal O})\subset{\mathcal M}$, the {\em
linear part} $L\xi$ is intrinsically defined as the $k={\mathcal
O}/{\mathcal M}$-linear map
$$
L\xi:{\mathcal M}/{\mathcal M}^2\rightarrow {\mathcal M}/{\mathcal
M}^2
$$
given by $f+{\mathcal M}^2 \mapsto \xi f+{\mathcal M}^2$. Note that
``elementary'' implies ``log-elementary''. Note also that a vector
field $\xi\in\mbox{Der}_k{\mathcal O}$ of the form
\begin{equation}
\label{formulauno} \xi=\sum_{i=1}^nf_ix_i\frac{\partial}{\partial
x_i}, \quad b_i\in {\mathcal O},
\end{equation}
has a non-nilpotent linear part if and only if one of the
 $f_i$ is a unit in $\mathcal O$.

\subsection{Newton polyhedron}
Note that ${\mathbf z}={\mathbf x}$, since all the variables have
$\mathbb Q$-linearly independent values. Any element $f\in{\mathcal
O}$ can be expanded in a formal series
$$
f=\sum f_{{\mathbf a}}{\mathbf x}^{\mathbf a}; \quad f_{\mathbf
a}\in k.
$$
The {\em support of $f$} is defined by $\mbox{Supp}(f;{\mathbf
x})=\{{\mathbf a};f_{\mathbf a}\ne 0\}\subset{\mathbb Z}_{\geq
0}^n$. For a vector field $\xi\in\mbox{Der}_k{\mathcal O}$ written
as in formula (\ref{formulauno}), the support is $$
 \mbox{Supp}(\xi;{\mathbf
x})=\cup_{i=1}^n\mbox{Supp}(f_i;{\mathbf x}).$$ The {\em Newton
polyhedron\/} ${\mathcal N}(\xi;{\mathbf x})$ is the convex hull in
${\mathbb R}^n$ of the set $\mbox{Supp}(\xi;{\mathbf x})+{\mathbb
R}^n_{\geq 0}$.

The local foliation ${\mathcal L}_{\mathcal A}$ contains a
 vector field $\xi$ of the form (\ref{formulauno}) such that the coefficients
$f_i\in{\mathcal O}$ have no common factor in $\mathcal O$, that we
call an {\em ${\mathbf x}$-generator} of ${\mathcal L}_{\mathcal
A}$. To see this, take any $\eta\in{\mathcal L}_{\mathcal A}$, then
$(\prod_{i=1}^nx_i)\eta$ is of the form (\ref{formulauno}) and now
it is enough to divide by the $\gcd$ of the $f_i$.

 We
define the {\em Newton polyhedron \/} $ {\mathcal N}({\mathcal
L};{\mathbf x})$ by  $ {\mathcal N}({\mathcal L};{\mathbf
x})={\mathcal N}(\xi;{\mathbf x})$, where $\xi$ is an $\mathbf
x$-generator of ${\mathcal L}_{\mathcal A}$.
\begin{remark}
The Newton polyhedron ${\mathcal N}({\mathcal
L};{\mathbf x})$ has vertices in ${\mathbb Z}^n_{\geq 0}$. Since the
coefficients $f_i$ have no common factor (and ``a fortiori'' they
are free of a monomial common factor) the only ${\mathbf v}\in
{\mathbb R}^n_{\geq 0}$ such that
$$
{\mathcal N}({\mathcal L};{\mathbf x})\subset {\mathbf v}+{\mathbb
R}^n_{\geq 0}
$$
is ${\mathbf v}=0$. Thus, if ${\mathcal N}({\mathcal L};{\mathbf
x})$ has only one vertex $\mathbf v$, then ${\mathbf v}=0$ and the
vector field $\xi$ has a non-nilpotent linear part. This implies
that ${\mathcal L}_ {\mathcal A}$ is elementary.
\end{remark}
\subsection{The effect of a blow-up} Let ${\mathcal
A}'=({\mathcal O},{\mathbf z}')$ be obtained from ${\mathcal A}$ by
an $(i,s)$-{\em blow-up}. Recall that there are no dependent
variables and hence it is a combinatorial blow-up. If
$\nu(x_i)<\nu(x_s)$, we have $x'_s=x_s/x_i$ and $x'_s=x_s$, for
$s\ne j$. Consider the affine function $ \sigma_{is}^i:{\mathbb
R}^n\rightarrow {\mathbb R}^n $ defined by
$$
\sigma_{is}^i({\mathbf a})_t=\left\{
\begin{array}{lr}
a_i+a_j,&\mbox{ if } t= s\\
a_s,&\mbox{ if } t\ne s
\end{array}
\right.
$$
Take ${\mathbf v}\in {\mathbb R}^n_{\geq 0}$ such that
$\sigma_{ij}^i\left({\mathcal N}({\mathcal L};{\mathbf x})\right)$
is inscribed in the orthant ${\mathbf v}+{\mathbb R}^n_{\geq 0}$.
Then the Newton polyhedron ${\mathcal N}({\mathcal L};{\mathbf x}')$
is obtained as
$$
{\mathcal N}({\mathcal L};{\mathbf
x}')=\left(\sigma_{is}^i\left({\mathcal N}({\mathcal L};{\mathbf
x})\right)-{\mathbf v}\right)+{\mathbb R}_{\geq 0}^n.
$$
In fact, the behavior of the Newton polyhedron is the same one as
the behavior of the Newton polyhedron of the ideal generated by the
coefficients $f_i$. In the case $\nu(x_i)>\nu(x_s)$, we do the same
argument with the corresponding affine map $\sigma_{is}^s$.
\subsection{End of the proof of Proposition \ref{pro:cuatro}.} We
can use the same idea as in the proof of Proposition \ref{prop:dos}.
Let $N$ be the number of vertices of ${\mathcal N}({\mathcal
L};{\mathbf x})$. After doing an $(i,j)$-blow-up, we obtain that
$N'\leq N$. If $N=1$, we are done. Assume that $N\geq 2$. Take two
distinct vertices ${\mathbf a}$ and ${\mathbf b}$ of ${\mathcal
N}({\mathcal L};{\mathbf x})$ and let $\mathbf v$ be the element in
${\mathbb Z}_{\geq 0}^n$ such that the set $\{{\mathbf a},{\mathbf
b}\}$ is inscribed in ${{\mathbf v}}+{\mathbb R}_{\geq 0}^n$. In
other terms, the monomial ${\mathbf x}^{\mathbf v}$ is the $\gcd$ of
${\mathbf x}^{\mathbf a}$ and ${\mathbf x}^{\mathbf b}$. Put
$\tilde{\mathbf a}={\mathbf a}-{\mathbf v}$ and $\tilde{\mathbf
b}={\mathbf b}-{\mathbf v}$. Note that for any index $t$ we have
$\tilde a_t\tilde b_t=0$ and also $\tilde {\mathbf a}\ne 0\ne \tilde
{\mathbf b}$. Choose indices $i,s$ with $\tilde a_i\tilde b_s\ne0$.
Do the $(i,s)$-blow-up. Assuming that $N'=N$, the set of indices
$$
\{t;\; \tilde a_t\ne 0 \mbox{ or } \tilde b_t\ne 0\}
$$
is contained in the corresponding one after blow-up. If the two sets
coincide, the amount $\tilde a_i+\tilde b_s$ decreases strictly.
This ends the proof.
\begin{remark}
The same kind of combinatorial game, but using centers of any
codimension, with a ``permissibility'' additional condition, is
called the Weak Hironaka's Game \cite{Spi}.
\end{remark}

\section{The Newton-Puiseux Polygon}
\label{secciontres}
 Let us assume in this section that
 $r=n-1$, $\kappa_\nu=k$ and take a
 parameterized regular local model ${\mathcal A}=({\mathcal
O},{\mathbf z}=({\mathbf x},{ y}))$. Note that since $r=n-1$, there
is only one dependent variable $y$.

Consider an element $f \in y^{-1}{\mathcal O}$, that we write
$f=\sum_{s=-1}^\infty h_{s}({\mathbf x})y^s$, where $h_{s}({\mathbf
x})$ is a formal series $h_{s}({\mathbf x})\in k[[{\mathbf x}]]\cap
{\mathcal O}$. The {\em Newton-Puiseux support of $f$} is the set
$$
\mbox{NPSup}({f};{\mathbf x},y)=\{(\nu(h_{s}),s); h_{s}\ne 0
\}\subset {\Gamma}\times{\mathbb Z}_{\geq -1}.$$ We denote by
$\alpha(f;{\mathbf x},y)$ the minimum abscissa of the Newton Puiseux
support, that is $\alpha(f;{\mathbf x},y)=\min\{(\nu(h_s))\}$. The
{\em main height } $\hbar(f;{\mathbf x},y)$ is the minimum of the
$s$ such that $\nu(h_s)=\alpha(f;{\mathbf x},y)$. Let
$\delta(f;{\mathbf x},y)$ be the minimum of the values
$\nu(h_s)+s\nu(y)$. The {\em critical segment} ${\mathcal
C}(f;{\mathbf x},y)$ is the set of  the $s$ such that
$$\nu(h_s)+s\nu(y)=\delta(f;{\mathbf x},y).$$ The {\em main height}
$\chi(f;{\mathbf x},y)$ is the highest $s$ in the critical segment.
Let us note that $\chi(f;{\mathbf x},y)\leq \hbar (f;{\mathbf
x},y)$.

Consider a finite list ${\mathbf f}=(f_1,f_2,\ldots,f_t)$ of
elements
 $f_j\in y^{-1}{\mathcal O}$. The Newton-Puiseux support
 $\mbox{NPSup}({\mathbf f};{\mathbf x},y)$ is the set of $(u,s)$,
  where $u$ is the minimum of the $u_j$ such
  that $(u_j,s)\in \mbox{NPSup}(f_j;{\mathbf x},y)$, for
  $j=1,2,\ldots,t$. We obtain in this way a definition for
  $\alpha({\mathbf f};{\mathbf x},y)$,
  $\hbar({\mathbf f};{\mathbf x},y)$
  ,$\delta({\mathbf f};{\mathbf x},y)$ and
  $\chi({\mathbf f};{\mathbf x},y)$ since these invariants depend
  only on the Newton-Puiseux support.

\subsection{Newton-Puiseux Polygon of a foliation} Consider the free
${\mathcal O}$-module $\mbox{Der}_k{\mathcal O}[\log {\mathbf x}]$
whose elements are the vector fields of the form
\begin{equation}
\label{camponmenosuno} \xi=\sum_{i=1}^{n-1}f_i({\mathbf
x},y)x_i\frac{\partial}{\partial x_i}+g({\mathbf
x},y)\frac{\partial}{\partial y}
\end{equation}
where $g\in{\mathcal O}$, $f_i\in{\mathcal O}$, $i=1,2,\ldots,n-1$.
Such vector fields will be called {\em $\mathbf x$-logarithmic
vector fields}, or simply {\em $\mathbf x$-vector fields}. Let us
denote $f_n=g/y$ and ${\mathbf f}=(f_1,f_2,\ldots,f_n)$. We define
$\mbox{NPSup}(\xi;{\mathbf x},y)=\mbox{NPSup}({\mathbf f};{\mathbf
x},y)$ and
\begin{eqnarray*}
  \alpha(\xi;{\mathbf x},y)=\alpha({\mathbf f};{\mathbf
  x},y);&\quad
  \hbar(\xi;{\mathbf x},y)= \hbar({\mathbf f};{\mathbf x},y)\\
  \delta(\xi;{\mathbf x},y)= \delta({\mathbf f};{\mathbf x},y);&\quad
  \chi(\xi;{\mathbf x},y)=\chi({\mathbf f};{\mathbf x},y).
\end{eqnarray*}
Given a foliation ${\mathcal L}\subset \mbox{Der}_kK$, we consider
the local {\em $\mathbf x$-logarithmic foliation ${\mathcal
L}_{\mathcal A}[\log{\mathbf x}]$ at $\mathcal A$} defined by
\begin{equation}
\label{eq:log} {\mathcal L}_{\mathcal A}[\log{\mathbf x}]={\mathcal
L}\cap \mbox{Der}_k{\mathcal O}[\log{\mathbf x}].
\end{equation}
We define the {\em main height} $\hbar({\mathcal L};{\mathcal A})$,
respectively the {\em critical height} $\chi({\mathcal L};{\mathcal
A})$, to be the minimum of the $\hbar({\xi};{\mathbf x}, y)$,
respectively $\chi({\xi};{\mathbf x}, y)$, where $\xi\in {\mathcal
L}_{\mathcal A}[\log{\mathbf x}]$. Note that $$\hbar({\mathcal
L};{\mathcal A})\geq\chi({\mathcal L};{\mathcal A})\geq -1.$$ These
ones are the main invariants we shall use to control the singularity
of ${\mathcal L}$ after performing a Puiseux package.

\subsection{The initial parts}
\label{subsectioninitialparts}
 Consider an element $h=\sum_{\mathbf m}\lambda_{\mathbf m}{\mathbf x}^{\mathbf m}\in {\mathcal O}\cap k[[{\mathbf
 x}]]$. Since
the values $\nu(x_i)$, $i=1,2,\ldots, n-1$ are $\mathbb Q$-linearly
independent, there is exactly one exponent ${\mathbf m}_0$ such that
$\nu(\lambda_{{\mathbf m}_0}{\mathbf x}^{{\mathbf m}_0})=\nu(h)$.
Moreover, if $\tilde h=h-\lambda_{{\mathbf m}_0}{\mathbf
x}^{{\mathbf m}_0}$ then $\nu(\tilde h)>\nu(h)$. Take an element
$\gamma\in\Gamma$ with $\gamma\leq\nu(h)$. We define the
$\gamma$-{\em initial form\/} $\mbox{In}_{\gamma}(h)$ by
$\mbox{In}_{\gamma}(h)=0$ if $\gamma<\nu(h)$ and $
\mbox{In}_{\nu(h)}(h)=\lambda_{{\mathbf m}_0}{\mathbf x}^{{\mathbf
m}_0}$ if $\gamma=\nu(h)$. Given a list ${\mathbf
h}=(h_1,h_2,\ldots,h_n)$ of elements $h_j=h_j({\mathbf x})\in
k[[x]]\cap {\mathcal O}$, and $\gamma\in\Gamma$ with
$\gamma\leq\min\{\nu(h_j(\mathbf x));j=1,2,\ldots,n\}$ we put
$$
\mbox{In}_\gamma({\mathbf h};{\mathbf x})=(\mbox{In}_\gamma({
h_1};{\mathbf x}),\mbox{In}_\gamma({ h_2};{\mathbf
x}),\ldots,\mbox{In}_\gamma({ h_n};{\mathbf x})).
$$
If we have a vector field of the form
$$
\eta= \sum_{j=1}^{n-1}h_{j}({\mathbf x})x_j\frac{\partial}{\partial
x_j} +h_{n}({\mathbf x})y\frac{\partial}{\partial y}
$$
and  $\gamma\leq\min\{\nu(h_j(\mathbf x));j=1,2,\ldots,n\}$ we put
$$\mbox{In}_\gamma(\eta;{\mathbf x})=
\sum_{j=1}^{n-1}\mbox{In}_\gamma(h_{j};{\mathbf
x})x_j\frac{\partial}{\partial x_j} +\mbox{In}_\gamma(h_{n};{\mathbf
x})y\frac{\partial}{\partial y}.
$$

 Take an $\mathbf x$-vector field $\xi\in \mbox{Der}{\mathcal O}[\log
{\mathbf x}]$  that we write as in equation (\ref{camponmenosuno}).
Put  $ f_j=\sum_{s=-1}^{\infty}h_{js}({\mathbf x})y^{s}$,
$j=1,2,\ldots,n$. We have $\xi=\sum_{s=-1}^ny^s\eta_s$, where
\begin{equation}
\label{eq:descomposicion} \eta_s=\sum_{j=1}^{n-1}h_{js}({\mathbf
x})x_j\frac{\partial}{\partial x_j} +h_{ns}({\mathbf
x})y\frac{\partial}{\partial y}; \quad s=-1,0,1,\ldots .
\end{equation}
 Put
$\delta=\delta(\xi;{\mathbf x},y)=\min_{j,s}\{\nu(y^sh_{js}({\mathbf
x}))\}$. We define the initial form $\mbox{In}(\xi;{\mathbf x},y)$
as
$$
\mbox{In}(\xi;{\mathbf x},y)=\sum_{s=-1}^\infty y^s
\mbox{In}_{\delta-s\nu(y)}(\eta_s;{\mathbf x}).
$$
Let us note that if $\tilde\xi=\xi-\mbox{In}(\xi;{\mathbf x},y)$,
then $\delta(\tilde\xi;{\mathbf x},y)>\delta(\xi;{\mathbf x},y)$.
Note also that if $\chi=\chi(\xi;{\mathbf x},y)$ is the critical
height, then $\mbox{In}_{\delta-s\nu(y)}(\eta_s;{\mathbf x})=0$ for
$s>\chi$ and $\mbox{In}_{\delta-\chi\nu(y)}(\eta_\chi;{\mathbf
x})\ne0$.
 In particular $\mbox{In}(\xi;{\mathbf x},y)$ is a finite sum
$$
\mbox{In}(\xi;{\mathbf x},y)=\sum_{s=-1}^\chi y^s
\mbox{In}_{\delta-s\nu(y)}(\eta_s;{\mathbf x}).
$$
Now we are going to give a particular expression of
$\mbox{In}(\xi;{\mathbf x},y)$ in terms of the contact rational
function $\Phi=y^d/{\mathbf x}^{\mathbf p}$.

Let us take an index $s$ such that
$\mbox{In}_{\delta-s\nu(y)}(\eta_s;{\mathbf x})\ne 0$ and in
particular $s\leq \chi$. Write $
\mbox{In}_{\delta-s\nu(y)}(\eta_s;{\mathbf x})= {\mathbf
x}^{{\mathbf q}(s)}\Lambda_s$, where $\Lambda_s$ is the linear
vector field
$$
\Lambda_s= \sum_{j=1}^{n-1}\lambda_{js}x_j\frac{\partial}{\partial
x_j}+\lambda_{ns}y\frac{\partial}{\partial y}
 $$
and  ${\mathbf q}(s)\in{\mathbb Z}_{\geq 0}^{n-1}$. Put ${\mathbf
r}(s)={\mathbf q}(s) - {\mathbf q}(\chi)$. We have
$$
\nu({\mathbf x}^{{\mathbf
r}(s)})=(\chi-s)\nu(y)=\frac{\chi-s}{d}\nu({\mathbf x}^{\mathbf
p}),$$ this implies that $ ((\chi-s)/{d}){\mathbf p}={\mathbf r}(s)
$ and thus $ ((\chi-s)/{d}){\mathbf p}\in {\mathbb Z}^{n-1}$. Since
the coefficients $p_1,p_2,\ldots,p_{n-1}$ have no common factor, we
have that $(\chi-s)/{d}\in {\mathbb Z}$. Put $t=(\chi-s)/d\in
{\mathbb Z}_{\geq 0}$; note that $t\leq\varrho$, where
$\varrho\in{\mathbb Z}_{\geq 0}$ is the biggest integer bounded
above by $(\chi+1)/d$.

  We may
write $\mbox{In}(\xi;{\mathbf x},y)$ as follows:
\begin{equation}
\label{eq:formainicial} \mbox{In}(\xi;{\mathbf x},y)={\mathbf
x}^{{\mathbf q}(\chi)}y^\chi
\sum_{s=-1}^{\chi}{\frac{1}{\Phi^{(\chi-s)/d}}}\Lambda_s= {\mathbf
x}^{{\mathbf q}(\chi)}y^\chi
\sum_{t=0}^{\varrho}{\frac{1}{\Phi^{t}}}\Lambda_{\chi-dt}.
\end{equation}
In order to simplify the notation, let us rename
$\Delta_{t}=\Lambda_{\chi-dt}$. Then
\begin{equation}
\label{eqtres} {\mathbf x}^{-{\mathbf
q}(\chi)}y^{-\chi}\Phi^{\varrho}\mbox{In}(\xi;{\mathbf x},y) =
\sum_{t=0}^{\varrho}{{\Phi}}^{\varrho-t}\Delta_t.
\end{equation}
We recall that $\Delta_0\ne 0$. \subsection{The expression of the
derivatives after a Puiseux package} Assume that ${\mathcal
A}'=({\mathcal O}',{\mathbf z}'=({\mathbf x}',{y}'))$ has been
obtained from ${\mathcal A}=({\mathcal O},{\mathbf z}=({\mathbf
x},{y}))$ by a Puiseux package. Let $\Phi=y^d/{\mathbf x}^{\mathbf
p}$ be the contact rational function.
 By remark \ref{obsecuaciones} we
have that $y'=\Phi-c$  and there is a matrix $B=(b_i^s)$ with
determinant $1$ or $-1$ and positive integer coefficients such that
$$
 z_s=\left(
 \prod_{i=1}^{n-1}{{x}_i'}^{{b}^{s}_i}\right)\Phi^{b^{s}_n};\quad
 s=1,2,\ldots,n.
$$
Moreover if $p_s=0$ we know that $x_s=x'_s$, that is $b^s_i=0$ if
$i\ne s$  and $b^s_s=1$. This implies that
\begin{eqnarray}
x'_i\frac{\partial}{\partial x'_i}&=&\sum_{s=1}^{n-1}{
b}_i^{s}x_s\frac{\partial}{\partial x_s}+
{b}_i^{n}y\frac{\partial}{\partial y};\quad i=1,2,\ldots,n-1,
\\
\Phi\frac{\partial}{\partial y'}&=&
\sum_{s=1}^{n-1}b^s_nx_s\frac{\partial}{\partial x_s}+
b^n_ny\frac{\partial}{\partial y}
\end{eqnarray}
Let $B^{-1}=(\tilde b^i_s)$ be the inverse matrix of $B=( b^s_i)$.
We obtain
\begin{eqnarray}
x_s\frac{\partial}{\partial x_j}&=&\sum_{i=1}^{n-1}{\tilde
b}_s^{i}x'_i\frac{\partial}{\partial x'_i}+
{\tilde b}_s^{n}\Phi\frac{\partial}{\partial y'};\quad s=1,2,\ldots,n-1,\\
y\frac{\partial}{\partial y}&=& \sum_{i=1}^{n-1}\tilde
b^i_nx'_i\frac{\partial}{\partial x'_i}+\tilde
b^n_n\Phi\frac{\partial}{\partial y'}
\end{eqnarray}
Note that the $\tilde b^s_i$ are integer (may be negative) numbers.
Moreover, we have
\begin{equation}
\label{eq:coeficiented} \tilde
b^n_n=\frac{1}{\Phi}y\frac{\partial}{\partial y}(\Phi)=d\ne 0.
\end{equation}
Finally, a given linear vector field
$\Delta=\sum_{i=1}^n\mu_iz_i\partial/\partial z_i$, we have
\begin{equation}
\label{eq:tildemu} \Delta=
\left\{\sum_{i=1}^{n-1}\tilde\mu_{i}x'_j\frac{\partial}{\partial
x'_j} +\tilde\mu_{n}y'\frac{\partial}{\partial y'}\right\}+
c\tilde\mu_{n}\frac{\partial}{\partial y'}.
\end{equation}
where $(\tilde
\mu_{1},\tilde\mu_{2},\ldots,\tilde\mu_{n})=(\mu_{1},\mu_{2},\ldots,\mu_{n})B^{-1}$.
\section{Rational co-rank one}
In this section we also assume that
 $r=n-1$, $\kappa_\nu=k$. We take a
 parameterized regular local model ${\mathcal A}=({\mathcal
O},{\mathbf z}=({\mathbf x},{ y}))$ and a foliation ${\mathcal
L}\subset\mbox{Der}_kK$.  We will prove the following result
\begin{proposition}
\label{pro:cinco} There is a a parameterized regular local model
${\mathcal A}'$ obtained from $\mathcal A$ by a  finite sequence of
coordinate changes in the dependent variable and coordinate blow-ups
with codimension two centers, such that one of the following
properties holds:
\begin{enumerate}
\item  There is $\hat f\in
\widehat {\mathcal O}'$ having transversal maximal contact with
$\nu$.\item The local foliation ${\mathcal L}_{{\mathcal A}'}$ is
non-singular if $n=2$ and elementary if $n\geq 3$.
\end{enumerate}
\end{proposition}
Now, Proposition \ref{pro:cinco} is a consequence of the following
five lemmas.
\begin{lemma}
\label{lema:uno}
 Assume that  ${\mathcal A}'$ has been obtained from $\mathcal A$ by
 a coordinate change in the dependent variable.
 Then $\hbar({\mathcal L};{\mathcal A}')=\hbar({\mathcal L};{\mathcal A})$.
\end{lemma}
\begin{proof} Left to the reader.
\end{proof}
\begin{lemma}
\label{lema:dos} Assume that ${\mathcal A}'$ has been obtained from
${\mathcal A}$ by a Puiseux package. Then $\hbar({\mathcal
L};{\mathcal A}')\leq \chi({\mathcal L};{\mathcal A})$. Moreover, we
have  $$ \hbar({\mathcal L};{{\mathcal A}'})<\chi({\mathcal
L};{{\mathcal A}})
$$
if $\chi({\mathcal L};{\mathcal A})\geq 1$ and $d({\mathcal A})\geq
2$, where $d({\mathcal A})$ is the
 the ramification index of $\mathcal A$.
\end{lemma}

\begin{lemma}
\label{lema:tres} Assume that $\hbar({\mathcal L};{\mathcal
A})\in\{-1,0\}$. We have the following properties
\begin{enumerate}
\item
If $\hbar({\mathcal L};{\mathcal A})=-1$, after performing a finite
sequence of coordinate blow-ups in the independent variables, we
obtain ${\mathcal A}'$ such that ${\mathcal L}_{{\mathcal A}'}$ is
non-singular.
\item If $\hbar({\mathcal L};{\mathcal
A})=0$, after performing a finite sequence of coordinate blow-ups in
the independent variables, we obtain ${\mathcal A}'$ such that
${\mathcal L}_{{\mathcal A}'}$ is elementary.
\end{enumerate}
\end{lemma}

\begin{lemma}
\label{lema:cuatro}
 If $n=2$, after performing a finite sequence of Puiseux packages we obtain
${\mathcal A}'$ such that either ${\mathcal L}_{{\mathcal A}'}$ is
non-singular or there is $\hat f\in \widehat {\mathcal O}'$  having
transversal maximal contact with $\nu$.
\end{lemma}
\begin{lemma}
\label{lema:cinco}
 Assume that  $\hbar({\mathcal L}_{\mathcal A})\geq
1$ and that the following property holds:
\begin{quote}
``After any finite sequence of coordinate blow-ups in the
independent variables, Puiseux packages and coordinate changes in
the dependent variable we have that $d({\mathcal A})=1$ and
$\hbar({\mathcal L}_{{\mathcal A}'})=\hbar({\mathcal L}_{{\mathcal
A}})$''.
\end{quote}
Then there is  $\hat f\in \widehat {\mathcal O}$ having transversal
maximal contact with $\nu$.
\end{lemma}
In order to show that Lemmas \ref{lema:uno}, \ref{lema:dos},
\ref{lema:tres}, \ref{lema:cuatro} and \ref{lema:cinco}  imply
Proposition \ref{pro:cinco}, let us only recall that $\chi({\mathcal
L};{\mathcal A})<\hbar({\mathcal L};{\mathcal A})$. So, unless we
have a transversal maximal contact, we arrive to the situation of
Lemma \ref{lema:tres} by a repeated application of Lemma
\ref{lema:dos} and we are done.

 Let us prove the above lemmas.
\subsection{The effect of a Puiseux package}
\label{subsectionpuiseuxeffect}
 Let us consider
${\mathcal A}'=({\mathcal O}',{\mathbf z}'=({\mathbf x}',y'))$
obtained from ${\mathcal A}$ by a Puiseux package. Take an $\mathbf
x$-vector field $\xi\in {\mathcal L}_{\mathcal A}[\log {\mathbf x}]$
such that $\chi(\xi;{\mathbf x},y)=\chi({\mathcal L};{\mathcal A})$
and let us write $\xi=\sum_{s=-1}^ny^s\eta_s$ as in  equations
(\ref{camponmenosuno}) and (\ref{eq:descomposicion}). In order to
simplify the notation, put $\chi=\chi(\xi;{\mathbf x},y)$ and
$\delta=\delta(\xi;{\mathbf x},y)$. Moreover, we denote
$d=d({\mathcal A})$ the ramification index associated to $\mathcal
A$. Let us write
 $\tilde\xi=\xi-\mbox{In}(\xi;{\mathbf x},y)$. We recall that
$\delta(\tilde\xi;{\mathbf x},y)>\delta$.

 Next we express
$\mbox{In}(\xi;{\mathbf x},y)$ and $\tilde \xi$ in terms of the
coordinates ${\mathbf z}'=({\mathbf x}',{ y}')$.
\begin{lemma}
\label{lema:seis} $\alpha(\tilde \xi;{\mathbf x}',y')>\delta$.
\end{lemma}
\begin{proof} Left to the reader.
\end{proof}
Let us consider now  $\mbox{In}(\xi;{\mathbf x},y)$ and let us
express it in the coordinates ${\mathbf z}'$. Let us recall equation
\ref{eqtres}, where $ {\mathbf x}^{-{\mathbf
q}(\chi)}y^{-\chi}\Phi^{\varrho}\mbox{In}(\xi;{\mathbf x},y) =
\sum_{t=0}^{\varrho}{{\Phi}}^{\varrho-t}\Delta_t$ and
$$\Delta_{t}=\Lambda_{\chi-dt}=\sum_{i=1}^n\lambda_{i,\chi-dt}z_i\partial/\partial z_i=\sum_{i=1}^n\mu_{it}z_i\partial/\partial z_i,$$
with $\Delta_0\ne 0$. Let us put $
\zeta=\sum_{t=0}^{\varrho}{{\Phi}}^{\varrho-t}\Delta_t$. We can
write $\zeta=\sum_{s\geq \beta'}y'^s\vartheta_s$, where
$\vartheta_{\beta'}\ne 0$ and all the $\vartheta_s$ are ${\mathbf
z}'$-linear vector fields $
\vartheta_s=\sum_{j=1}^n\alpha_{js}z'_j{\partial}/\partial z'_j$.

\begin{lemma}
\label{lema:siete}We have $\beta'\leq \chi$. If $\chi\geq
 1$ and $d\geq 2$, then $\beta'<\chi$.
\end{lemma}
\begin{proof} Looking at the equation \ref{eq:tildemu}, we see that $\zeta=\sum_{t=0}^{\varrho}({y'+c})^{\varrho-t}\Delta_t$
and
$$
\Delta_t=\sum_{j=1}^n\tilde\mu_{jt}z'_j\frac{\partial}{\partial
z'_j}+c \tilde\mu_{nt}\frac{\partial}{\partial y'};\quad
(\tilde\mu_{1t},\tilde\mu_{2t},\ldots,\tilde\mu_{nt})=
(\mu_{1t},\mu_{2t},\ldots,\mu_{nt})B^{-1}.
$$
Let $\varsigma=\max\{t;\; \Delta_t\ne
 0\}\leq \varrho$. Then
 $
\zeta=  \Phi^{\varrho-\varsigma}
\sum_{t=0}^{\varsigma}{{\Phi}}^{\varsigma-t}\Delta_t$. Recalling
that $\Phi=y'+c$, and dividing the above expression by
$\Phi^{\varrho-\varsigma}$, we obtain $\beta'\leq \varsigma$.
Remember that $ \varrho$ is the greatest integer bounded above by
$(\chi+1)/d$. Then, if $d\geq 2$ and $\chi\geq 2$, or $d\geq 3$ and
$\chi=1$ we obtain $\beta'\leq\varsigma\leq \varrho<\chi$. If
$\chi=0$ and $d\geq 2$ we have $\varrho=0$ and then $\beta'\leq 0$.
It remains to study the cases with $d=1$, the case $d=2$, $\chi=1$
and the case $\chi=-1$.

{\em The case $\chi=-1$.\/} In this case $\varrho=\varsigma=0$. In
particular $\zeta=\Delta_0$. Moreover
$\Delta_0=\Lambda_{-1}=\mu_{n0}y\partial/\partial y$. Recalling that
$\tilde b_n^n=d$ in view of equation \ref{eq:coeficiented}, we have
$$
\zeta=\mu_{n0}y\partial/\partial y=\mu_{n0} \sum_{j=1}^{n}\tilde
b^j_nz'_j\frac{\partial}{\partial
z'_j}+\mu_{n0}dc\frac{\partial}{\partial y'}.
$$
This implies that $\alpha_{n,-1}=\mu_{n0}dc\ne 0$ and thus
$\beta'=-1$.

{\em Cases with $d=1$, $\chi\geq 0$.\/} We reason by contradiction,
assuming that $\beta'\geq \chi+1$. This implies that
$\varsigma=\varrho=\chi+1$. In particular, we have
$\Delta_{\chi+1}\ne 0$ and $\Delta_{\chi+1}=\Lambda_{-1}$. Note that
$\Lambda_{-1}=\mu y\partial/\partial y$, where
$\mu=\mu_{n,\chi+1}=\lambda_{n,-1}$. Now, our contradiction
hypothesis $\beta'\geq\chi+1$ implies that $\zeta(y')$ is divisible
by $y'^{\chi+2}$.
 We have
\begin{eqnarray*}
\zeta(y')&=&\sum_{t=0}^{\chi+1}\Phi^{\chi+1-t}\Delta_t(y')=\\
&=&\Phi\left(\Phi^{\chi+1}\tilde\mu_{n0}+\Phi^{\chi}\tilde\mu_{n1}+
\Phi^{\chi-1}\tilde\mu_{n2}+\cdots+
\Phi\tilde\mu_{n\chi}+\tilde\mu_{n,\chi+1}\right).
\end{eqnarray*}
Recall that $\Phi=y'+c$, then we necessarily have that
$\zeta(y')=\tilde\mu_{n0}y'^{\chi+2}$, since the biggest possible
power of $y'$ in the above expression is $y'^{\chi+2}$ and its
coefficient is $\tilde\mu_{n0}$. Moreover we also have that
$\Phi=y'+c$ divides $\zeta'(y')$. The only possibility is that
$\zeta(y')=0$ and hence all the coefficients $\tilde\mu_{nt}$ are
zero, for $t=1,2,\ldots,\chi+1$. This is a contradiction, since $
\tilde\mu_{n,\chi+1}=\tilde b_n^{n}\mu=d\mu\ne 0$.

{\em Case $d=2$, $\chi=1$}. Let us reason by contradiction, assuming
that $\beta'\geq \chi$. Then $\varsigma=\varrho=\chi=1$. We have
$\zeta=\Phi\Delta_0+\Delta_1$ and $y'^2$ must divide $\zeta(y')$.
That is
$$
\zeta(y')=\Phi\left(\Phi\tilde\mu_{n0}+\tilde\mu_{n1}\right)
=y'^{2}\tilde\mu_ {n0}.
$$
We deduce as above that $\zeta(y')=0$ and thus
$\tilde\mu_{n1}=\tilde\mu_{n0}=0$. Note that $0\ne\Delta_1$, since
$\varsigma=1$. Moreover, in our case $\Delta_t=\Lambda_{1-2t}$ and
thus $\Delta_1=\Lambda_{-1}=\mu y\partial/\partial y\ne 0$. Now we
have $\tilde\mu_{n1}=2\mu$ and we obtain that $\tilde\mu_{n1}\ne0$
and $\tilde\mu_{n1}=0$ simultaneously, contradiction.
\end{proof}
\begin{lemma}
\label{lema:ocho}
 $\hbar({\mathcal L};{\mathcal A}')\leq \beta'$.
\end{lemma}
\begin{proof} It is enough to show that $\beta'=\hbar(\xi;{\mathbf
x}',y')$. We have $\beta'=\hbar(\zeta;{\mathbf x}';y')$ and
$\alpha(\zeta;{\mathbf x}';y')=0$. Recall that
$\mbox{In}(\xi;{\mathbf x},y)={\mathbf x}^{{\mathbf
q}(\chi)}y^{\chi}\Phi^{-\varrho}\zeta$, where
$$
\nu({\mathbf x}^{{\mathbf
q}(\chi)}y^{\chi}\Phi^{-\varrho})=\nu({\mathbf x}^{{\mathbf
q}(\chi)}y^{\chi})=\delta.
$$
Moreover, in view of Remark \ref{obsecuaciones}, we have that
$$
{\mathbf x}^{{\mathbf q}(\chi)}y^{\chi}\Phi^{-\varrho}={{\mathbf
x}'}^{\mathbf q'}\Phi^{r'}, \quad \mbox{ and } \nu({{\mathbf
x}'}^{\mathbf q'})=\delta.
$$
Noting that $\mbox{In}(\xi;{\mathbf x},y)={{\mathbf x}'}^{\mathbf
q'}\Phi^{r'}\zeta$, we deduce that $\alpha(\mbox{In}(\xi;{\mathbf
x},y);{\mathbf x}',y')=\delta$ and $\hbar(\mbox{In}(\xi;{\mathbf
x},y);{\mathbf x}',y')=\hbar(\zeta;{\mathbf x}',y')=\beta'$.
Moreover, by Lemma \ref{lema:seis}, we have $$
\delta=\alpha(\mbox{In}(\xi;{\mathbf x},y);{\mathbf
x}',y')<\alpha(\tilde\xi;{\mathbf x}',y').
$$
Recalling that $\xi=\mbox{In}(\xi;{\mathbf x},y)+\tilde\xi$, we have
that $\alpha(\xi;{\mathbf x}',y')=\delta$ and $$ \hbar(\xi;{\mathbf
x}',y')=\hbar(\mbox{In}(\xi;{\mathbf x},y);{\mathbf x}',y')=\beta'.
$$
This ends the proof.
\end{proof}
\begin{remark}
Lemma \ref{lema:uno} follows from Lemma \ref{lema:siete}, in view of
Lemma \ref{lema:ocho}.
\end{remark}
Before giving a proof of Lemma \ref{lema:tres}, we explain the
effect of the blow-ups in the independent variables in the following
result.
\begin{lemma}
\label{lema:nueve} Given $\mathcal A$ and $\mathcal L$, after
performing finitely many coordinate blow-ups in the independent
variables with centers of codimension two, we can obtain ${\mathcal
A}'$ such that $\alpha({\mathcal L};{\mathcal A}')=0$. Moreover
$\hbar({\mathcal L};{\mathcal A}')\leq \hbar({\mathcal L};{\mathcal
A})$.
\end{lemma}
\begin{proof}Write $\xi=\sum_{s=-1}^\infty y^s\eta_s$ with $
\eta_s=\sum_{j=1}^nh_{js}({\mathbf x})z_j{\partial}/{\partial z_j}$.
Let us do a  blow-up in the independent variables and let ${\mathbf
x}',y$ be the obtained variables. Then the same decomposition as
above acts in this new set of variables, that is
$\xi=\sum_{s=-1}^\infty y^s\eta_s$ where we can write
$$
\eta_s=\sum_{j=1}^nh'_{js}z'_j\frac{\partial}{\partial z'_j};\quad
h'_{js}\in k[[{\mathbf x}']].
$$
Moreover the ideal $I'_s\subset k[[{\mathbf x}']]$ generated by
$\{h'_{js}\}_{j=1}^n$ is $ I'_s=I_sk[[{\mathbf x}]]$, where $I_s $
is the ideal of $k[[{\mathbf x}]]$ generated by
$\{h_{js}\}_{j=1}^n$. This already implies that
$$
\alpha(\eta_s;{\mathbf x})=\alpha(\eta_s;{\mathbf x}');\quad
s=-1,0,1,\ldots .
$$
In particular we have that $\hbar(\xi;{\mathbf
x},y)=\hbar(\xi;{\mathbf x}',y)$.

Moreover, the ideal $I'=\sum_{s=-1}^\infty I'_s\subset k[[{\mathbf
x}']]$ is also given by  $I'=Ik[[{\mathbf x}']]$, where
$I=\sum_{s=-1}^\infty I_s\subset k[[{\mathbf x}]]$. Thus, we can
apply classical results of reduction of singularities under
combinatorial blow-ups, that can be proved as in Proposition
\ref{pro:cuatro} (see also \cite{Can-R-S}) to assure that after a
finite number of blow-ups in the independent variables with centers
of codimension two, the ideal $I$ is generated by a single monomial,
say ${{\mathbf x}'}^{\mathbf q'}$. We obtain an ${\mathbf
x}'$-vector field $\xi'={{\mathbf x}'}^{-\mathbf q'}\xi\in {\mathcal
L}_{{\mathcal A}'}[\log {\mathbf x}']$ such that
$\alpha(\xi';{\mathbf x}',y)=0$.
\end{proof}
\begin{remark} In the above lemma we have
$\hbar({\mathcal L};{\mathcal A}')=\hbar({\mathcal L};{\mathcal
A}')$. Anyway, we do not need to use this fact.
\end{remark}

 We obtain an immediate proof of Lemma \ref{lema:tres}. By
 Lemma \ref{lema:nueve}, we may suppose that there is
 $\xi\in {\mathcal L}_{\mathcal A}[\log {\mathbf x}]$ such that
 $\hbar(\xi;{\mathbf x},y)=\hbar({\mathcal L};{\mathcal A})$ and
 $\alpha(\xi;{\mathbf x},y)=0$. Now,
 it is evident that
 \begin{enumerate}
 \item If $\hbar(\xi;{\mathbf x},y)=-1$, then $\xi$ is non-singular.
 \item If $\hbar(\xi;{\mathbf x},y)=0$, then $\xi$ is elementary (or non-singular).
 \end{enumerate}

\subsection{Getting a formal hypersurface of transversal maximal
contact}
 Let us give a proof of Lemma \ref{lema:cinco}. In view of Lemma
 \ref{lema:nueve}, after performing finitely many blow-ups in the independent variables,
 we can assume that there is
 $\xi\in {\mathcal L}_{\mathcal A}[\log {\mathbf x}]$ such that
 $\hbar(\xi;{\mathbf x},y)=\hbar({\mathcal L};{\mathcal A})$ and
 $\alpha(\xi;{\mathbf x},y)=0$. Moreover, we also have that
 $$
\chi(\xi;{\mathbf x},y)=\hbar(\xi;{\mathbf x},y),
 $$
since otherwise, an application of Lemma \ref{lema:dos} allows us to
decrease $\hbar({\mathcal L};{\mathcal A})$. Moreover, in view of
our hypothesis, we have $d({\mathcal A})=1$ and $\hbar({\mathcal
L};{\mathcal A})\geq 1$.

\begin{lemma}
\label{lema:diez} Let $\Phi=y/{\mathbf x}^{\mathbf p}$ be the
contact rational function. We have ${\mathbf p}\in{\mathbb Z}_{\geq
0}^{n-1}$.
\end{lemma}
\begin{proof} Let us keep the notations of
subsection \ref{subsectioninitialparts}.  Recall
that
$$
\mbox{In}(\xi;{\mathbf x},y)=\sum_{s=-1}^\chi y^s{\mathbf
x}^{{\mathbf q}(s)}\Lambda_s.
$$
Since $\alpha(\xi;{\mathbf x},y)=0$ and $\chi(\xi;{\mathbf
x},y)=\hbar(\xi;{\mathbf x},y)$, we have  ${\mathbf q}(\chi)=0$.
Thus, for any $s$ such that $\Lambda_s\ne 0$ we have
$$
\nu({\mathbf x}^{{\mathbf q}(s)})=(\chi-s)\nu(y)
$$
and hence $(\chi-s){\mathbf p}={{\mathbf q}(s)}$. Noting that
${\mathbf q}(s)\in {\mathbb Z}_{\geq 0}^{n-1}$, it is enough to show
that there is at least an index $s<\chi$ such that $\Lambda_s\ne 0$.
 Assume the contrary. Then
$$
\mbox{In}(\xi)=y^{\chi}\Lambda_{\chi},
$$
where $\chi=\hbar(\xi;{\mathbf x};y)\geq 1$. Let us do a Puiseux
package, taking the notations of the proof of Lemma
\ref{lema:siete}, we obtain $\varsigma=0$ and hence
$\chi'\leq\beta'\leq \varsigma=0$. Contradiction.
\end{proof}

In this situation, we have  $\nu(y-c{\mathbf x}^{\mathbf
p})>\nu(y)$. Let us do the coordinate change $ y'=y-c{\mathbf
x}^{\mathbf p}$.  The situation repeats. In this way we can produce
a sequence of elements $y^{(j)}\in {\mathcal M}\setminus {\mathcal
M}^2$, such that $y^{(0)}=y$ and
$$ y^{(j)}=y^{(j-1)}-c_{j}{\mathbf
x}^{{\mathbf p}(j)};\quad
\nu(y^{(j)})>\nu({y^{(j-1)}}),\,j=1,2,\ldots .$$
 Taking $\hat
f=\lim_j y^{(j)}$, we obtain the desired formal hypersurface.

\subsection{The case of dimension two}
The statement of Lemma \ref{lema:cuatro} is a consequence of
Seidenberg's reduction of singularities in dimension two \cite{Sei}.
Let us see this. Assuming that we do not get non-singular points,
after finitely Puiseux packages, we obtain a ``simple singularity''
in the sense of Seidenberg. It is given by an $x$-vector field of
the form
$$
\xi=(\lambda+a(x,y))x\frac{\partial}{\partial x}+ (\alpha x+\mu
y+\tilde b(x,y))\frac{\partial}{\partial x};\quad a(0,0)=0,\;\tilde
b(x,y)\in {\mathcal M}^2
$$
where, $(\lambda,\mu)\ne (0,0)$ and if $\lambda\ne 0$ then
$\mu/\lambda\notin {\mathbb Q}_{>0}$. Such singularity has exactly
two formal invariant curves: $x=0$ and $\hat f=0$, where $\hat
f=y-\hat\phi(x)$. They are non-singular and transversal one to the
other. After doing one more blow-up, the exceptional divisor is
invariant and we obtain exactly two simple singularities, one of
them corresponds to the strict transform of $x=0$, it is a {\em
corner}, and the other one is in the strict transform of $\hat f=0$.
This shows that blowing-up a corner produces only corners as
singularities, thus, since the valuation has rational rank one and
we have nontrivial Puiseux packages, we necessarily do blow-ups
outside the corners. Hence we follow the infinitely near points of
$\hat f=0$. ``A fortiori'', we obtain that $\hat f$ is non-algebraic
(otherwise the value of $\hat f$ would be infinite) and has maximal
contact with $\nu$.
\section{Etale Puiseux packages}
\subsection{Review on etale neighborhoods}
 Let us recall the definition of a
local etale morphism as one can see in \cite{Art}. Let us fix the
local ring ${\mathcal O}={\mathcal O}_{M,P}$ of a projective model
$M$ of $K$ at the center $P$ in $M$ of the $k$-valuation $\nu$ of
$K$ and assume that $P$ is a regular point of $M$. Here we assume
that $\nu$ is a real valuation with $\kappa_\nu=k$.

Consider a morphism ${\mathcal O}\rightarrow \widetilde{\mathcal O}$
of local rings. We say that ${\mathcal O}\rightarrow
\widetilde{\mathcal O}$ is {\em local-etale} or that
$\widetilde{\mathcal O}$ is a {\em local-etale extension } of
${\mathcal O}$ if we have the following properties:
\begin{enumerate}
\item The local rings ${\mathcal O}$ and $\widetilde{\mathcal O}$ have the
same residual field.
\item $\widetilde{\mathcal O}$ is the localization at a prime ideal of an
etale ${\mathcal O}$-algebra.
\end{enumerate}
An etale $\mathcal O$-algebra is an $\mathcal O$-algebra of the type
$B={\mathcal O}[t_1,t_2,\ldots,t_n]/(f_1,f_2,\ldots,f_n)$, where the
Jacobian matrix of the $f_i$ is invertible in $B$. This is
equivalent to say that $B$ is a finitely generated $A$-flat algebra
and $\Omega^1_AB=0$. Note that $\widetilde{\mathcal O}$ is also a
regular local ring and its fraction field $\widetilde K$ is a
finitely generated algebraic extension of $K$. Recall also that
$\widetilde{\mathcal O}\subset{\mathcal O}^h\subset\widehat{\mathcal
O}$, where ${\mathcal O}^h$ is the henselian closure of ${\mathcal
O}$.

We say that the pair $(\widetilde{\mathcal O},\tilde\nu)$ is a {\em
local etale extension of $({\mathcal O},\nu)$} if
$\widetilde{\mathcal O}$ is a local-etale extension of ${\mathcal
O}$ and $\tilde\nu$ is a $k$-valuation of $\widetilde K$ centered at
$\widetilde{\mathcal O}$ such that $\tilde\nu\vert_K=\nu$.
 Note that  $\tilde \nu$ is a real $k$-valuation and
$\kappa_{\tilde\nu}=k$.

In the following proposition, we summarize the properties that allow
us to work
 ``up to local-etale extensions''.
\begin{proposition}
\label{prop:seis}
 Consider a foliation ${\mathcal
L}\subset\mbox{Der}_kK$ and a real $k$-valuation $\nu$ of $K$ such
that $\kappa_\nu=k$. Let  $(\widetilde{\mathcal O},\tilde\nu)$ is a
be a local etale extension of $({\mathcal O},\nu)$ and denote
 $\widetilde{\mathcal L}=\widetilde K{\mathcal
L}\subset \mbox{Der}_k\widetilde K$  the induced foliation on
$\widetilde K$. Assume that we respectively have:
\begin{enumerate}
\item The foliation  $\widetilde{\mathcal L}$ is log-elementary at
$\widetilde{\mathcal O}$.
\item There is a formal $\hat f\in \widehat{\mathcal O}$ with
transversal maximal contact relatively to $\widetilde{\mathcal O}$.
\end{enumerate}
Then, up to perform a finite sequence of local blow-ups of $\mathcal
O$ we respectively have:
\begin{enumerate}
\item The foliation  ${\mathcal L}$ is log-elementary at
${\mathcal O}$.
\item There is a formal $\hat f\in \widehat{\mathcal O}$ with
transversal maximal contact relatively to ${\mathcal O}$.
\end{enumerate}
\end{proposition}
\begin{proof} Let $\tilde x_1,\tilde x_2, \ldots, \tilde x_n$ be a regular system of parameters of
$\widetilde{\mathcal O}$. Consider $\tilde h=\prod_{i=1}^n\tilde
x_i$. The ideal $\tilde h\widetilde{\mathcal O}$ gives a principal
ideal $h{\mathcal O}={\mathcal O}\cap \tilde h\widetilde{\mathcal
O}$. We can do the local uniformization of $h$ by using centers that
respect the fact that $\widetilde{\mathcal L}$ is log elementary
(relatively to $\tilde{\mathbf x}$) (see \cite{Can} to the
definition of permissible centers and the needed properties).
Finally we get that $h$ is a monomial and we are done.
\end{proof}
\subsection{Etale Puiseux packages}

We introduce here an etale version of Puiseux packages for the case
$r=1$. It has the same effect over a foliation as the Puiseux
packages introduced in Section \ref{seccion:uno}, but it will allow
us to do an accurate control of the foliation. Indeed, the study of
the case $n=3$, $r=1$ will be done under the use of etale Puiseux
packages.

We assume that $\nu$ is a valuation with rational rank $r=1$ and
$\kappa_\nu=k$.

Let ${\mathcal A}=({\mathcal O},{\mathbf z}=({x},{\mathbf y}))$ be a
regular parameterized model. Consider a dependent variable $y_j$.
Let $\Phi=y_j^d/{ x}^{p}$ be the contact rational function and $c\in
k$ such that $\nu(\Phi-c)>0$. Recall that $d$ is the {\em
$y_j$-ramification index of} ${\mathcal A}$.

\begin{remark}
In the case that $d=1$,
 all
the blow-ups in a Puiseux package are ``in the first chart'' in the
sense that we always have $\nu(y_j)\geq \nu(x)$.
\end{remark}

Let us consider the ring  ${\mathcal O}^\natural={\mathcal
O}[T]/(T^d-x)={\mathcal O}[t]$, where $T$ is an indeterminate and
let $\widetilde K$ be the fraction field of ${\mathcal O}^\natural$.
We know \cite{Vaq} that there are $k$-valuations $\tilde\nu$ of
$\widetilde K$ such that $R_{\tilde \nu}\cap K=R_\nu$. Note that all
the $\tilde\nu$ have the same group of values. Let us choose one of
them, say $\tilde\nu$. The ring ${\mathcal O}^\natural$ is a regular
local ring that supports a parameterized regular local model
$$
{\mathcal A}^\natural=({\mathcal O}^\natural,{\mathbf
z}^\natural=(t,{\mathbf y}))
$$
relative to $\tilde K$ and $\tilde \nu$. We have $k\subset {\mathcal
O}\subset {\mathcal O}^\natural$ and ${\mathcal
M}^\natural\cap{\mathcal O}={\mathcal M}$. Moreover,
$k=\kappa_{\tilde \nu}$ and $t^d=x$. Let us note that
$\tilde\nu(y/t^{\mathbf p})=0$. In particular $d(\widetilde{\mathcal
A})=1$. Let $\tilde c\in k$
 be such that $\tilde\nu(y/t^{\mathbf p}-\tilde c)>0$, we see that $\tilde
 c^d=c$.
 \begin{definition} We say that $(\widetilde{\mathcal A},\tilde \nu)$ has been
 obtained from $({\mathcal A},\nu)$ by an {\em etale $j$-Puiseux package}
 if and only if  $\widetilde{\mathcal A}$ has been obtained from
 ${\mathcal A}^\natural$
by a $j$-Puiseux package. \end{definition}

\begin{proposition}
\label{prop:siete} Assume that  $\widetilde{\mathcal
A}=(\widetilde{\mathcal O},\tilde{\mathbf z}=(t,\tilde{\mathbf y}))$
has been
 obtained from ${\mathcal A}$ by an {etale $j$-Puiseux package}.
 There is ${\mathcal A}'=({\mathcal
O}',{\mathbf z}'=(x',{\mathbf y}'))$ obtained from $\mathcal A$ by a
 $j$-Puiseux package such that $(\widetilde{\mathcal O},\tilde\nu)$ is
  a local-etale extension of
 $
 ({\mathcal
 O}',\nu)$.
\end{proposition}
\begin{proof} Consider the $j$-Puiseux package
${\mathcal A}^\natural\mapsto \widetilde{\mathcal A}$.  Put
$\Phi=y_j^d/x^p$ and $\tilde\Phi=y_j/t^p$, the respective contact
rational functions for ${\mathcal A}$ and ${\mathcal A}^\natural$.
Note that $\tilde\Phi^d=\Phi$. Let $c,\tilde c\in k$ be such that
$\tilde\nu(\tilde\Phi-\tilde c)>0$ and $\tilde c^d=c$. We have
$$
\tilde y_j=\tilde\Phi -\tilde c;\quad y'_j=\Phi-c.
$$
Moreover $\tilde y_j$ is a simple root of a polynomial over
${\mathcal O}'$ as the following relation shows
$$
y'_j=(\tilde y_j+\tilde c)^d-c.
$$
Now $\tilde t$ is of the form $\tilde t=x'P(\tilde y_j)$, where
$P(0)\ne 0$. This is enough to obtain the conclusion.
\end{proof}
\begin{remark} If
$\widetilde{\mathcal A}=(\widetilde{\mathcal O},\tilde{\mathbf
z}=(t,\tilde{\mathbf y}))$ has been obtained from ${\mathcal
A}=({\mathcal O}, {\mathbf z}=(x,{\mathbf y}))$ by an etale Puiseux
package, then ${\mathcal O}\subset\widetilde{\mathcal O}$ and
$t^d=x$.
\end{remark}
\begin{definition} We say that
$({\mathcal A},\nu)\mapsto (\widetilde{\mathcal A},\tilde\nu)$ is an
{\em etale standard transformation} if and only if
$(\widetilde{\mathcal A},\tilde\nu)$ has been obtained from
$({\mathcal A},\nu)$ by an etale Puiseux Package or by a coordinate
change in the dependent variables.
\end{definition}

\section{Rational rank one}
We end here the proof of Theorem \ref{teo:tres}. To do this we
consider the following proposition
\begin{proposition}
\label{prop:ocho} Let ${\mathcal L}\subset\mbox{Der}_kK$ be a
foliation over $K$, where $n=3$. Take a $k$-valuation $\nu$ of $K$
of rational rank one and such that $\kappa_\nu=k$. Assume that
${\mathcal A}$ is a parameterized regular local model for $K$ and
$\nu$. Then, there is a finite sequence of etale standard
transformations
$$
({\mathcal A},\nu)=({\mathcal A}_0,\nu_0)\mapsto ({\mathcal
A}_1,\nu_1)\mapsto \cdots\mapsto({\mathcal A}_N,\nu_N)=({\mathcal
A}',\nu')
$$
such that either the transformed foliation ${\mathcal L}'$ is
log-elementary in ${\mathcal A}'$ or there is $\hat f\in {\mathcal
O}'$ having transversal maximal contact.
\end{proposition}
 Proposition \ref{prop:ocho} gives the end of the
 proof of Proposition \ref{pro:tres} and hence it completes the proof of Theorem
\ref{teo:tres}. Indeed,   by propositions  \ref{pro:cuatro} and
\ref{pro:cinco} we obtain Proposition \ref{pro:tres} for rational
rank $r=2,3$. For the case of rational rank $r=1$ and $n=3$, we
obtain Proposition \ref{pro:tres} from Proposition \ref{prop:ocho}
in view of propositions \ref{prop:seis} and \ref{prop:siete}.

This section is devoted to the proof of Proposition \ref{prop:ocho}.
In all this section we assume implicitly that we do not get a formal
transversal maximal contact.

Recall that we in this section we have $n=3$, the rational rank of
$\nu$ is equal to one and $\kappa_\nu=k$. We start with a
parameterized regular local model $ {\mathcal A}=({\mathcal O};
{\mathbf z}=(x,w,y))$ and a foliation ${\mathcal
L}\subset\mbox{Der}_kK$.
\subsection{The independent coefficient}
\label{equispreparacion} Let $\xi$ be an $\mathcal O$-generator of
${\mathcal L}_{\mathcal A}[\log x]$. Let us put $H=\xi(x)/x\in
{\mathcal O}$. Consider an etale standard transformation $({\mathcal
A},\nu)\mapsto ({\mathcal A},\nu')$ where
$${\mathcal A}'=({\mathcal O}', {\mathbf z}'=(t, w',y')).$$
Recall that $t^d=x$, for $d\geq 1$. We know that ${\mathcal
L}'_{{\mathcal A}'}[\log t]$ is generated by a germ of vector field
of the form $ \xi'=t^{q}\xi$ where $ q\in{\mathbb Z}$. Moreover,  we
have that
$$
\xi(t)/t=\lambda' \xi(x)/x,
$$
where $\lambda'=1/d\in {\mathbb Q}_{>0}$. This implies that
\begin{equation}
\label{equno} H'=\xi'(t)/t=\lambda' t^{q}H\in {\mathcal O}'.
\end{equation}
In particular, the coefficient $H$ is transformed essentially ``as a
function'' under the etale standard transformations. This allows us
to obtain the following result
\begin{proposition}
\label{pro:nueve} After finitely many etale standard transformations
we can chose an $\mathcal O$-generator $\xi$ of ${\mathcal
L}_{\mathcal A}[\log x]$
 such that
 $
\xi(x)/x=\lambda x^m$,
 where $\lambda\in {\mathbb Q}_{>0}$.
\end{proposition}
\begin{proof} We apply to $H$ the usual local
uniformization for functions. We obtain that $H=x^mU$, where $U$ is
a unit. Now we divide $\xi$ by $U$.
\end{proof}
Moreover,  the above form of $H$ is persistent under etale standard
transformations. This justifies the next definition.
\begin{definition}
We say that ${\mathcal L}$ is {\em $x$-prepared relatively to
$\mathcal A$} if there is an $\mathcal O$-generator $\xi$ of
${\mathcal L}_{\mathcal A}[\log x]$ such that $\xi(x)/x=\lambda
x^q$, for $0\ne\lambda\in{\mathbb Q}$. Such generators $\xi$ will be
called {\em $x$-privileged generators}.
\end{definition}
In view of Proposition \ref{pro:nueve} we can obtain that ${\mathcal
L}$ is $x$-prepared after a finite number of etale-standard
transformations and this property is persistent under new
etale-standard transformations.
\subsection{Invariants from the Newton Puiseux Polygon}

Take $f\in w^{-1}k[[x,w]]$ that we write $f=\sum_{t=-1}^\infty
w^tf_t(x)$. We put
 \begin{equation}
\label{eq:invariantes}
 \lambda(f;x,w)=\min_t\{\nu(f_t(x))+t\nu(w)\};\quad
 \alpha(f;x,w)=\min_t\{\nu(f_t(x))\}.\end{equation}
Consider  a vector field $\eta$ of the form
\begin{equation}
\label{eq:xwvectorfields}
 \eta=
a(x,w)x\frac{\partial}{\partial x}+b(x,w)\frac{\partial}{\partial
w}+c(x,w)y\frac{\partial}{\partial y}.
\end{equation}
We denote
\begin{eqnarray}
\lambda(\eta;x,w)&=&\min\{\lambda(a;x,w),\lambda(b/w;x,w),\lambda(c;x,w)\}\\
\alpha(\eta;x,w)&=&\min\{\alpha(a;x,w),\alpha(b/w;x,w),\alpha(c;x,w)\}
\end{eqnarray}
Let us note that $\alpha(b/w;x,w)=\alpha(b;x,w)$. We also write
\begin{equation}
\Lambda(\eta;x,w)=\lambda(\eta;x,w)-\alpha(\eta;x,w).
\end{equation}
Note that $\Lambda(\eta;x,w)\geq -\nu(w)$.

 \begin{remark}
 \label{rk:diez}
We can draw a Newton-Puiseux polygon ${\mathcal N}$ for $f$, or for
$\eta$, by considering the support
$\{(\nu(f_t(x),t))\}\subset\Gamma\times{\mathbb Z}_{\geq -1}\subset
{\mathbb R}\times{\mathbb Z}_{\geq -1}$. Then $\alpha$ is the
abscissa of the highest vertex and $\lambda$ corresponds to the
smallest value  $a+\nu(w^b)$, where $(a,b)$ is in the support. In
particular, we have that $\Lambda=-\nu(w)$ if and only if $\mathcal
N$ has the single vertex $(\alpha,-1)$.
\end{remark}
Consider a vector field $\xi=\sum_{s=-1}^\infty y^s\eta_s\in
\mbox{Der}_k{\mathcal O}[\log x]$, where
\begin{equation}
\label{eq:etasubese} \eta_s= a_s(x,w)x\frac{\partial}{\partial
x}+b_s(x,w)\frac{\partial}{\partial
w}+c_s(x,w)y\frac{\partial}{\partial y}.
\end{equation}
 We denote
$\alpha(\xi;{\mathcal
A})=\alpha(\xi;x,w,y)=\min_{s=-1}^\infty\{\alpha(\eta_s;x,w)\}$. Let
us note that $\alpha(\xi;{\mathcal A})=0$ when $\xi$ is a generator
of ${\mathcal L}_{\mathcal A}[\log x]$, since $x$ is not a common
factor of the coefficients. The {\em main height
$\hbar(\xi;{\mathcal A})$} is the minimum of the $s$ such that
$\alpha(\eta_s;x,w)=\alpha(\xi;{\mathcal A})$. When $\xi$ is a
generator of $\mathcal L$, we put $\hbar({\mathcal L}, {\mathcal
A})=\hbar(\xi,{\mathcal A})$.

Denote
 $\delta(\xi;{\mathcal
A})=\min_{s=-1}^\infty \{\alpha(\eta_s;x,w)+s\nu(y)\}$. We say that
$s$ {\em belongs to the critical segment ${\mathcal
C}({\xi};{\mathcal A})$} if
$\alpha(\eta_s;x,w)+s\nu(y)=\delta(\xi;{\mathcal A})$. The {\em
critical height $\chi(\xi;{\mathcal A})$} is the greatest $s\in
{\mathcal C}({\xi};{\mathcal A})$. Note that $\chi(\xi;{\mathcal
A})\leq \hbar(\xi;{\mathcal A})$.

\begin{remark}
\label{rk:once} We can draw a Newton-Puiseux polygon ${\mathcal
N}(\xi;{\mathcal A})\subset {\mathbb R}\times {\mathbb Z}_{\geq -1}$
by taking as support the set
$$\{(\alpha(\eta_s;x,w),s);\; s=-1,0,1,\ldots\}.$$
Then $(\alpha(\xi;x,w,y),\hbar(\xi;x,w,y))$ is the main vertex of
the Newton-Puiseux polygon. We also have that $\delta(\xi;x,w,y)$ is
the smallest value $a+\nu(y^b)$, where $(a,b)$ is in the support.
Nevertheless, this Newton-Puiseux polygon needs to be prepared by
performing preliminary transformations in the variables $x,w$ in
order to be a useful tool in the control of the transformations in
the variables $x,y$.
\end{remark}
The invariants in three variables make sense also for
$f(x,w,y)=\sum_{s}y^sf_s(x,w)$. Thus, we write
\begin{equation}
\alpha(f;x,w,)=\min_s\{\nu(f_s;x,w)\};\quad
\delta(f,x,w,y)=\min_s\{\nu(f_s;x,w)+s\nu(y)\}.
\end{equation}
\subsection{Prepared situations in two variables} Take a vector field $\eta$ as in
equation (\ref{eq:xwvectorfields}).  We say that $\eta$  is {\em
$(x,w)$-prepared } if there is $q\in{\mathbb Z}_{\geq 0}$ such that
$$(a,b,c)=x^q(\tilde a(x,w),\tilde b(x,w),\tilde
c(x,w));\quad (\tilde b(0,0),\tilde c(0,0))\ne
 (0,0). $$ We say that $\eta$ is {\em dominant} if $\tilde b_s(0,0)\ne
 0$ and {\em recessive} if $\tilde b(0,0)=0,\tilde c (0,0)\ne 0$.

 \begin{remark}
\label{rkdos}
   The condition $\Lambda(\eta;x,w)=-\nu(w)$ is
 equivalent to say that $\eta$ is prepared-dominant. If $\eta$ is
 prepared-recessive, then $0\geq \Lambda(\eta;x,w)>\nu(w)$.
\end{remark}
\begin{definition}
\label{defstrongprepared} Take $\eta$ as in equation
(\ref{eq:xwvectorfields}).
 We
 say that $\eta$ is {\em strongly $(x,w)$-prepared} if there is a
 decomposition
\begin{equation}
 \eta=
 x^\rho U(x,w)\theta+x^\tau V(x,w)y\frac{\partial}{\partial y};\quad
 \theta=xh(x,w)x\frac{\partial}{\partial x}+\frac{\partial}{\partial w}
\end{equation}
satisfying the following properties
\begin{enumerate}
\item $\rho,\tau\in {\mathbb Z}\cup\{+\infty\}$, with $\rho\ne \tau$.
Here $\rho=+\infty$, respectively or $\tau=+\infty$, indicates that
$U(x,w)$, respectively $V(x,w)$, is identically zero.
\item We can write $U=\lambda+xf(x,w)$ and $V=\mu+xg(x,w)$, where $\lambda,\mu\in
k$. Moreover, if $\rho\ne +\infty$ then $\lambda\ne 0$ and if
$\tau\ne +\infty$ then $\mu\ne 0$.
\end{enumerate}
\end{definition}
 Let us note that ``strongly prepared'' implies
``prepared''. The dominant case corresponds to $r<t$ and the
recessive case to $r>t$.
\subsection{Effect of etale $w$-Puiseux packages} Let us perform an etale
$w$-Puiseux package and let $(t,w',y)$ be the obtained coordinates.
Recall that $t^d=x$ and $\nu(w/t^p)=0$, where $p,d$ are without
common factor. Moreover, we have $w'=w/t^p-c$, with $c\ne 0$, and
hence
\begin{equation} \label{explosincampo}
x\frac{\partial}{\partial x}= \frac{1}{d}\left\{
t\frac{\partial'}{\partial t}- p(w'+c) \frac{\partial'}{\partial w'}
\right\} ; \quad \frac{\partial}{\partial
w}=\frac{1}{t^p}\frac{\partial'}{\partial w'};\quad
\frac{\partial}{\partial y}= \frac{\partial'}{\partial y}.
\end{equation}
Consider  $\eta$ as in equation (\ref{eq:xwvectorfields})  and write
$$
\eta= a'(t,w')t\frac{\partial'}{\partial t}+
b'(t,w')\frac{\partial'}{\partial w'}+ c'(t,w')
y\frac{\partial'}{\partial y}
$$
in the coordinates $t,w',y$. Then we have
\begin{equation}
\begin{array}{ccccl}
\label{eqtransformacionnivel}  a'&=&\eta(t)/t&=&({1}/{d})a \\
b'&=&\eta(w')&=&(w'+c)\{b/w-(p/d)a\}=t^{-p}\{b-(p/d)t^p(w'+c)a\}\\
 c'&=&\eta(y)/y&=&c.
\end{array}
\end{equation}
From these considerations, we obtain the following results:
\begin{lemma}
 \label{lema:once} Consider
$f=\sum_{\ell=-1}^{\infty}w^\ell f_\ell(x)\in w^{-1}k[[x,w]]$. We
have
$$\alpha(f;t,w')=\lambda(f;x,w).$$ As a consequence, we also have
that $\alpha(\eta;t,w')=\lambda(\eta;x,w)$.
\end{lemma}
\begin{proof} Take a monomial
$w^ax^b$.  Note that $w^ax^b=t^{ap+bd}(w'+c)^{a}$ where $(w'+c)^{a}$
is a unit, hence
\begin{eqnarray*}
\lambda(w^ax^b;x,w)&=&\nu(w^ax^b)=\\
&=&\nu(t^{ap+bd}(w'+c)^{a})=\nu(t^{ap+bd})=\alpha(w^ax^b;t,w').
\end{eqnarray*}
Note that both $\lambda (_-;x,w)$ and $\alpha(_-;t,w')$ have the
usual valuative properties. This gives in particular that
$\alpha(f;t,w')\geq \lambda(f;x,w)$, as a consequence of the above
property for monomials. Put $\lambda=\lambda(f;x,w)$ and let us
decompose $f=\mbox{L}(f)+f^*$, where $\lambda(f^*;x,w)>\lambda$ and
$\mbox{L}(f)$ is of the form
$$
\mbox{L}(f)= \sum_{\nu(w^ax^b)=\lambda}\mu_{ab}w^ax^b.
$$
Now it is enough to prove that $\alpha(\mbox{L}(f);x',w')=\lambda$.
We know that if $\nu(w^ax^b)=\lambda$ then $m=ad+bp$ is independent
of $(a,b)$, since $\nu(t^{m})=\nu(x^aw^b)$.
 We have
$$
\mbox{L}(f)=t^{m}\sum \mu_{ab}(w'+c)^{a}\ne 0.
$$
Then $\alpha(\mbox{L};t,w')=\lambda$. The last statement comes from
the above arguments and the equations \ref{eqtransformacionnivel}.
\end{proof}
\begin{corollary}
\label{cor:desplazamiento}
 Consider $\eta$ as in equation
(\ref{eq:xwvectorfields}). We have $$\alpha(\eta;t,w')\geq
\alpha(\eta;x,w)-\nu(w),$$ and the equality holds exactly when
$\eta$ is $(x,w)$-prepared and dominant.
\end{corollary}
\begin{proof} We know that
$\alpha(\eta;t,w')=\lambda(\eta;x,w)$ by Lemma \ref{lema:once}. Now,
it is a direct consequence of the definitions that $
\lambda(\eta;x,w)\geq \alpha(\eta;x,w)-\nu(w)$, and the equality
holds exactly when $\eta$ is $(x,w)$-prepared and dominant, in view
of Remark \ref{rkdos}.
\end{proof}
\begin{lemma}
\label{lemaochoBIS} Consider $\eta$ as in equation
(\ref{eq:xwvectorfields}). We have
\begin{enumerate}
\item If $\Lambda (\eta;x,w)<0$, then
$\Lambda (\eta;t,w')=-\nu(w)$ and hence  $\eta$ is $(t,w')$-prepared
and dominant.
\item If $\eta$ is $(x,w)$-prepared and dominant, then $\eta$ is
also $(t,w')$-prepared and dominant.
\item If $\Lambda (\eta;x,w)=0$, then
$\Lambda (\eta;t,w')\leq 0$.
\item If $\eta$ is $(x,w)$-prepared and recessive, then $\eta$ is
$(t,w')$-prepared.
\end{enumerate}

\end{lemma}
\begin{proof}
 Write $\eta=x^{m}\eta'$, where
$\nu(x^m)=\alpha(\eta;x,w)$. If we substitute $\eta$ by $\eta'$ we
can assume without loss of generality that $\alpha(\eta;x,w)=0$ and
thus $\Lambda(\eta;x,w)=\lambda(\eta;x,w)$. Let us put $\overline
b=b/w=\sum_{\ell=-1}^\infty \overline b_\ell(x)w^\ell$. Consider
first the case that $\lambda(\eta;x,w)<0$. This implies that
\begin{equation}
\label{eq:ocho} \eta=\overline{b}_{-1}(x)\frac{\partial}{\partial
w}+\eta^*=x^rU(x)\frac{\partial}{\partial w}+\eta^*,\, U(0)\ne 0,
\end{equation}
where $0\leq \nu(b_{-1}(x))=\nu(x^r)=\nu(t^{rd})<\nu(w)=\nu(t^p)$
and $\eta^*$ has the form
$$
\eta^*=\sum_{\ell=0}^\infty
w^\ell\left\{a_{\ell}(x)x\frac{\partial}{\partial x}+
\overline{b}_{\ell}(x)w\frac{\partial}{\partial w}+
c_{\ell}(x)y\frac{\partial}{\partial y}
 \right\}
$$
 By equations
(\ref{explosincampo}) we see that $\alpha(\eta^*;t,w')\geq 0$ and
$$
x^rU(x)\frac{\partial}{\partial w}=
{t^{rd-p}}U(t^d)\frac{\partial}{\partial w'}.
$$
Note that $rd-p<0$. We obtain
$$\Lambda(\eta;t,w')=-\nu(w')$$ and thus
 $\eta$ is
$(t,w')$-prepared and dominant. This proves statement 1. Now,
statement 2  is a direct consequence of statement 1.

Assume that $\Lambda(\eta;x,w)=0$. Write $\eta$ as in Equation
(\ref{eq:ocho}), where $\nu(x^r)\geq \nu(w)$ (we accept the case
$r=+\infty$ to denote that $w$ divides $\eta(w)$). If
$\nu(x^r)=\nu(w)$, by the same argument as above  we obtain that
$\Lambda(\eta;t,w')=-\nu(w')$, hence $\eta$ is $(t,w')$-prepared and
dominant. If  $\nu(x^r)>\nu(w)$,  we have $\lambda(\eta^*;x,w)=0$.
Thus, we can write
$$
\eta= \left(\mu_1x\frac{\partial}{\partial
x}+\mu_2w\frac{\partial}{\partial w}+\mu_3y\frac{\partial}{\partial
y}\right)+\eta^{**},
$$
where $(\mu_1,\mu_2,\mu_3)\ne (0,0,0)$ and
$\lambda(\eta^{**};x,w)>0$. By equations
(\ref{eqtransformacionnivel}) we have
$$
\eta-\eta^{**}= \frac{\mu_1}{d}x'\frac{\partial'}{\partial x'}+
\frac{d\mu_2-p\mu_1}{d}(w'+c) \frac{\partial'}{\partial w'}+
\mu_3y\frac{\partial}{\partial y},
$$
where $\alpha(\eta^{**};t,w')=\lambda(\eta^{**};t,w')>0$. We obtain
\begin{eqnarray}
0= \alpha(\eta-\eta^{**};t,w')= \alpha(\eta;t,w')\\
0\geq \lambda(\eta-\eta^{**};t,w')=\lambda(\eta;t,w').
\end{eqnarray}
This ends the proof of statement 3. Note that if $d\mu_2-p\mu_1\ne
0$ then $\eta$ is $(t,w')$-prepared and dominant. If
$d\mu_2-p\mu_1=0$ and $\mu_3\ne 0$, we have that $\eta$ is
$(t,w')$-prepared and recessive. Now, if $\eta$ where
$(x,y)$-prepared and recessive, then $\mu_3\ne 0$. This proves
statement 4.
\end{proof}

\begin{proposition}
\label{pro:diez} Consider $\eta$ as in equation
(\ref{eq:xwvectorfields}). After performing finitely many etale
$w$-Puiseux packages, either we get transversal formal maximal
contact or we obtain one of the following properties:
\begin{enumerate}
\item[a)] The vector field $\eta$ is strongly $(x,w)$-prepared dominant and
 this property persists under new etale $w$-Puiseux packages.
\item[b)] The vector field  $\eta$ is strongly $(x,w)$-prepared recessive and
 this property persists under new etale $w$-Puiseux packages.
\end{enumerate}
\end{proposition}
\begin{proof}
 By the two dimensional
desingularization for vector fields \cite{Sei} and since we do not
get maximal contact, we can obtain  $\eta$ written down as
$$
\eta=f(x,w)\theta+g(x,w)y\frac{\partial}{\partial y};\quad
\theta=xh(x,w)x\frac{\partial}{\partial x}+\frac{\partial}{\partial
w}.
$$
Under new etale $w$-Puiseux packages, this form persist. Let us see
it. First, we know that
$$
\theta=\frac{1}{d}\left\{t^{d}h'(t,w')t\frac{\partial'}{\partial
t}+\left(\frac{d}{t^d}- {p(w'+c)t^{d}h'(t,w')}\right)
\frac{\partial'}{\partial w'}\right\}
$$
where $h'(t,w')=h(t^d,t^p(w'+c))$. This allows us to write
$\theta=t^{-d}W(t,w')\theta'$, where
$$
W(t,w')=1-t^{2d}(p/d)(w'+c)h'(t,w')
$$
is a unit and $\theta'$ has the same form as $\theta$. Note that
$W(t,w')-1$ is divisible by $t$. Now, we write
$$
\eta=f'(t,w')t^{-d}W(t,w')\theta'+g'(t,w')y\frac{\partial}{\partial
y},
$$
where $f'(t,w')=f(t^d,t^p(w'+c))$ and $g'(t,w')=g(t^d,t^p(w'+c))$.
By the standard desingularization of functions, we can perform new
etale $w$-Puiseux packages to obtain that
$$
f=x^{\rho}U(x,w);\; g=x^{\tau}V(x,w),
$$
where $U,V$  and $\rho,\tau$ satisfy to the properties in Definition
\ref{defstrongprepared} (note that it is possible that $\rho$ or
$\tau$ are negative; to recover a non-meromorphic vector field we
can multiply by a suitable power of $x$). By performing new etale
$w$-Puiseux packages, the difference $\nu(x^\tau)-\nu(x^\rho)$
increases the positive amount $\nu(w)$. If this difference is
positive, we are in case a), if it is always negative, we obtain
case b). If it is zero, in the next step it is positive.
\end{proof}
\begin{remark}
\label{rk:estabilidadbienpreparadodom}
 The above proof also shows that if $\eta$ is strongly
$(x,w)$-prepared and dominant,  it is so with respect to $(t,w')$.
Nevertheless, it is not always true that if $\eta$ is strongly
$(x,w)$-prepared and recessive the same holds with respect to
$(t,w')$. We start with $\rho>\tau$, by it can happen that
$\rho'\leq \tau'$ and in this case, after a new etale $w$-Puiseux
package we would obtain a dominant situation.
\end{remark}
\begin{remark}
\label{rk:desplazamientoabscisa} Assume that we are in one of the
situations a) or b) described in Proposition \ref{pro:diez}. Let us
perform an etale $w$-Puiseux package. Then we have
$$
\alpha(\eta;t,w')=\left\{\begin{array}{lr}\alpha(\eta;x,w)-\nu(w)&\mbox{
dominant case a)}\\
\alpha(\eta;x,w)&\mbox{ recessive case b)}
\end{array}\right.
$$
To see this, the only difficulty is the recessive case. Note that
since the recessive situation is stable under any finite sequence of
etale  $w$-Puiseux packages, we have $\nu(x^{\rho-\tau})>\nu(w)$,
and this is enough to assure the above formula.
\end{remark}

\subsection{Preparations in three variables.}
Consider a vector field $\xi\in \mbox{Der}_k{\mathcal O}[\log x]$,
that we write $\xi=\sum_{s=-1}^\infty y^s\eta_s$ where the $\eta_s$
are like in equation (\ref{eq:etasubese}).
\begin{definition}
\label{def:wellpreparation} Let $h=\hbar(\xi;{\mathcal A})$ be the
main height. We say that $\xi$ is {\em main-vertex prepared} with
respect to $\mathcal A$ when $\eta_h$ is $(x,w)$-prepared and
dominant. If in addition $\eta_h$ is strongly $(x,w)$-prepared, we
say that  $\xi$ is {\em strongly main-vertex prepared}. We say that
${\mathcal L}$ is {\em well prepared } with respect to $\mathcal A$
if there is   $\xi\in{\mathcal L}_{\mathcal A}[\log x]$  that is
$x$-prepared and strongly main-vertex prepared.
\end{definition}
\begin{remark}
\label{rktres}
 If $\xi$ is main-vertex prepared, we have
$\hbar(\xi;{\mathcal A})\geq 0$. Moreover, if $\alpha(\xi;{\mathcal
A})=0$ and $\xi$ is main-vertex prepared with
 $\hbar(\xi;{\mathcal A})=0$, then $\xi$ is a non-singular vector field.
\end{remark}

\begin{proposition} Assume that $\xi$ is $x$-prepared and strongly main-vertex prepared
 with respect to $\mathcal A$. Let us perform an etale $w$-Puiseux
package to obtain ${\mathcal A}'=({\mathcal O}', (t,w',y))$. Then
$\xi$ is $t$-prepared, strongly main-vertex prepared with respect to
${\mathcal A}'$ and the main height does not vary, that is
$\hbar(\xi;{\mathcal A}')=\hbar(\xi;{\mathcal A})$.
\end{proposition}
\begin{proof} The fact that $\xi$ is $t$-prepared has been proved in subsection
\ref{equispreparacion}. The decomposition $\xi=\sum_{s=-1}^\infty
y^s\eta_s$ is the same one with respect to $x,w,y$ and with respect
to $t,w',y$. Let us put $h=\hbar({\mathcal L};{\mathcal A})$. By
hypothesis $\eta_h$ is strongly $(x,w)$-prepared and dominant and
hence it is also strongly $(t,w')$-prepared and dominant, in view of
 Remark \ref{rk:estabilidadbienpreparadodom}.
 Now we have only to show that $h$
is also the main height  $\hbar(\xi;t,w',y)$ relatively to
${\mathcal A}'$. By Corollary \ref{cor:desplazamiento} we have
$$
\alpha(\eta_h;t,w')=\alpha(\eta_h;x,w)-\nu(w),
$$
since $\eta_h$ is $(x,w)$-prepared and dominant. For any other index
$s$ we have
$$
\alpha(\eta_s;t,w')\geq \alpha(\eta_s;x,w)-\nu(w),
$$
and this is enough to see that $\eta_h$ gives the main height for
$\xi$ with respect to ${\mathcal A}'$.
\end{proof}
\begin{proposition}
\label{prop:doce} Assume that $\xi=\sum_{s=-1}^\infty y^s\eta_s$ is
$x$-prepared and strongly main-vertex prepared with respect to
$\mathcal A$. Let us put $h=\hbar(\xi;x,w,y)$.  By performing
finitely many etale $w$-Puiseux packages we have the following
properties:
\begin{enumerate}
\item
 For any $s<h$, the vector field $\eta_s$ is
strongly $(x,w)$-prepared and this is stable, with the same
character dominant or recessive, under any new finite sequence of
etale $w$-Puiseux packages.
\item
The critical segment ${\mathcal C}(\xi;{\mathcal A})$ does not vary
under any new finite sequence of etale $w$-Puiseux packages, and all
the levels $s\in {\mathcal C}(\xi;{\mathcal A})$ have the same
character dominant or recessive.
\end{enumerate}
\end{proposition}
\begin{proof} The first statement is a corollary of Proposition \ref{pro:diez}.
Let us prove the second statement, assuming that statement 1 holds.
Let us perform an etale $w$-Puiseux package. In view of Remark
\ref{rk:desplazamientoabscisa} we have
$$
\alpha(\eta_s;t,w')=\left\{\begin{array}{lr}
 \alpha(\eta_s;x,w)-\nu(w)&\mbox{ dominant case}\\
 \alpha(\eta_s;x,w)& \mbox{ recessive case}
\end{array} \right.
$$
Thus, the critical segment thus not vary if $\eta_s$ is dominant for
all $s\in {\mathcal C}(\xi;{\mathcal A})$. If there is an $s_0\in
{\mathcal C}(\xi;{\mathcal A})$ such that  $\eta_{s_0}$ is dominant,
then all the recessive $s$ in the critical segment disappear under a
new $w$-Puiseux package and we are in the first case. Finally, if
under any finite sequence of etale $w$-Puiseux packages there is no
dominant $\eta_s$ that appears in the critical segment, the elements
in the critical segment are also stable.
\end{proof}
\begin{definition}
We say that $\xi=\sum_{s=-1}^\infty y^s\eta_s$ is {\em completely
prepared} with respect to $\mathcal A$ if it is $x$-prepared,
strongly main-vertex prepared and the properties 1 and 2 of
Proposition \ref{prop:doce} hold. We have
 two possible situations:
\begin{enumerate}
\item [a)] {\em Dominant critical segment}. The $\eta_s$
corresponding to $s$ in the critical segment are strongly
$(x,w)$-prepared and dominant.
\item [b)] {\em Recessive critical segment}. The $\eta_s$
corresponding to $s$ in the critical segment are strongly
$(x,w)$-prepared and recessive.
\end{enumerate}
\end{definition}
\begin{remark}
\label{rk:alturaprincipalycritica}
  Assume that $\xi=\sum_{s=-1}^\infty y^s\eta_s$ is
completely prepared with respect to $\mathcal A$. Let
$\chi=\chi(\xi;{\mathcal A})$ be the critical height and
$h=\hbar(\xi;{\mathcal A})$ the main height. We have $\chi\leq h$.
Moreover, since $\eta_h$ is strongly $(x,w)$-prepared and dominant,
in the case of a recessive  critical segment we have $\chi\leq h-1$.
\end{remark}
\subsection{Critical initial part and critical polynomial} Let us consider a
vector field $\xi=\sum_{s=-1}^\infty
y^s\eta_s\in\mbox{Der}_k({\mathcal O})[\log x]$ and write it as
$$\xi=\sum_{s=-1}^\infty \sum_j y^sx^j\left\{ a_{sj}(w)x\frac{\partial}{\partial
x}+b_{sj}(w)\frac{\partial}{\partial
w}+c_{sj}(w)y\frac{\partial}{\partial y}\right\}.$$ We know that
$$
\delta(\xi;x,w,y)=\min\left\{\nu(x^jy^s);\quad
(a_{sj}(w),b_{sj}(w),c_{sj}(w))\ne (0,0,0)\right\}.
$$
Put $\delta= \delta(\xi;x,w,y)$. We define the {\em critical initial
part } of $\xi$ by
\begin{equation}
\mbox{Crit}(\xi;x,w,y)=\sum_{\nu(x^jy^s)=\delta} y^sx^j\left\{
a_{sj}(w)x\frac{\partial}{\partial
x}+b_{sj}(w)\frac{\partial}{\partial
w}+c_{sj}(w)y\frac{\partial}{\partial y}\right\}.
\end{equation}
Obviously, if we put $\xi^*=\xi-\mbox{Crit}(\xi;x,w,y)$ we have
$\delta(\xi^*;x,w,y)>\delta(\xi;x,w,y)$.


\begin{definition}
Take $\delta\in\Gamma$.
 A monic polynomial $P(x,y)\in k[x,y]$ given by
 $$
 P(x,y)=y^m+\lambda_{m-1}x^{n_1}y^{m-1}+\lambda_{m-2}x^{n_2}y^{m-2}+\cdots+\lambda_0x^{n_m}
 $$
is called {\em $\nu$-homogeneous or degree $\delta$} if and only if
$\nu(x^{n_j}y^{m-j})=\nu(y^m)=\delta$ for any $j$ such that
$\lambda_j\ne 0$. It is called a {\em Tchirnhausen polynomial} if
$\lambda_{m-1}=0$.
\end{definition}
\begin{remark}
\label{rk:poltchirnhausen} Let us perform an etale $y$-Puiseux
package, to obtain coordinates $t,w,y'$ such that $t^d=x$ and
$y'=y/t^p-c$. Consider a monic $\nu$-homogeneous polynomial
$P=P(x,y)$ or degree $\delta$. Then
$$
P=P(t^d,t^p(y'+c))=t^{q'} P(1,y'+c),
$$
where $\nu(t^{q'})=\delta$ and $P(1,y'+c)$ is a monic polynomial of
degree $m$ in the variable $y'$. Write $ P(1,y'+c)=y'^{h'} Q (y')$,
where $Q(0)\ne 0$.  We have $h'\leq m$. Moreover, the only
possibility to have $m=h'$ is that $ P(x,y)=(y-cx^{n_1})^m$. This
cannot occur when $P(x,y)$ is a Tchirnhausen polynomial. Hence if
$P(x,y)$ is a Tchirnhausen polynomial we have $h'<m$. This argument
is crucial in most of the procedures of reduction of singularities
in characteristic zero.
\end{remark}
\begin{lemma}
\label{lema:criticalinicialpolynomial} Assume that
$\xi=\sum_{s=-1}^\infty y^s\eta_s$ is completely prepared relatively
to $\mathcal A$. Then, the critical initial part
$\xi_0=\mbox{Crit}(\xi;x,w,y)$  satisfies that
\begin{enumerate}
\item $\xi_0(x)=\xi_0(y)=0$, in the case of dominant critical
segment. \item $\xi_0(x)=\xi_0(w)=0$, in the case of recessive
critical segment.
\end{enumerate}
More precisely, the critical initial part $\xi_0$ takes one of the
following forms
\begin{eqnarray}
\label{ecuacioncasoa} \xi_0= \lambda x^q\sum_{s=0}^\chi\lambda_s y^s
x^{q_s} \frac{\partial}{\partial
w};\quad\mbox{ dominant critical segment case }\\
\label{ecuacioncasob} \xi_0= \lambda x^q\sum_{s=-1}^\chi
\lambda_sy^s x^{q_s} y\frac{\partial}{\partial
y};\quad\mbox{recessive critical segment case }
\end{eqnarray}
where $\lambda\ne 0$, $\lambda_\chi=1$ and
$\nu(y^sx^{q_s})=\delta-\nu(x^q)$ for each $s$ with $\lambda_s\ne
0$.
\end{lemma}
\begin{proof} Put $h=\hbar(\xi;x,w,y)$. Recall that for any $s\leq
h$ the vector field
$$
\eta_s= \sum_j x^j\left\{ a_{sj}(w)x\frac{\partial}{\partial
x}+b_{sj}(w)\frac{\partial}{\partial
w}+c_{sj}(w)y\frac{\partial}{\partial y}\right\}= \sum_j
x^j\eta_{sj}
$$
is $(x,w)$-strongly prepared. Put $\alpha_s=\alpha(\eta_s;x,w)$ and
let us take $r_s$  such that $\nu(x^{r_s})=\alpha_s$. Write $
\eta_s=x^{r_s}\tilde\eta_s$. In view of definition
\ref{defstrongprepared}, we have that
$$
\tilde\eta_s=\left\{\begin{array}{cc}\mu_s\partial/\partial
w+x\overline\eta_s
&\mbox{(dominant case)}\\
\mu_sy\partial/\partial y+x\overline\eta_s&\mbox{(recessive case)}
\end{array}
\right.
$$
We end by putting $\lambda_s=\mu_s/\lambda$ if $s$ is in the
critical segment and $\lambda_s=0$ otherwise.
\end{proof}
\begin{definition} In the situation of Lemma
\ref{lema:criticalinicialpolynomial}, we define the {\em critical
polynomial} $P_\xi(x,y)$ of $\xi$ with respect to $x,w,y$ to be
$$
P_\xi(x,y)=\left\{\begin{array}{cccc} \xi_0(w)/\lambda
x^q&=&\sum_{s=0}^\chi\lambda_sy^sx^{q_s}&\mbox{(dominant critical segment)}\\
 \xi_0(y)/\lambda
x^q&=&\sum_{s=-1}^\chi\lambda_sy^{s+1}x^{q_s}&\mbox{(recessive
critical segment)}
\end{array}
\right.
$$
{\rm(It is a $\nu$-homogeneous monic polynomial of $\nu$-degree
$\chi \nu(y)$, respectively $(\chi+1)\nu(y)$, in the case of a
dominant, respectively recessive critical segment.)}
\end{definition}
\begin{remark}
\label{rk:formulainitialpart}
 The critical initial part is obtained from the critical polynomial by
the formula
$$
\mbox{Crit}(\xi;x,w,y)=\left\{\begin{array}{cc}\lambda
x^qP_\xi(x,y)\partial/\partial w
&\mbox{(dominant critical segment)}\\
\lambda x^qP_\xi(x,y)\partial/\partial y&\mbox{(recessive critical
segment)}
\end{array}\right.
$$
\end{remark}

\subsection{Stability of the main height. Dominant critical segment}
\label{seccionestabilidaddominante} In this subsection we start to
study the effect of an etale $y$-Puiseux package on the main height.
Let us consider $\xi\in \mbox{Der}_k{\mathcal O}[\log x]$ and denote
$$
h=\hbar(\xi;x,w,y);\quad \chi=\chi(\xi;x,w,y);\quad
\delta=\delta(\xi;x,w,y);\quad \xi_0=\mbox{Crit}(\xi;x,w,y).
$$
Let us perform an etale $y$-Puiseux package,  to obtain $t,w,y'$
such that $t^d=x$ and $y'=y/t^p-c$. We recall that
\begin{equation}
\label{eq:paqueteney} x\frac{\partial}{\partial
x}=\frac{1}{d}\left\{t\frac{\partial'}{\partial
t}-p(y'+c)\frac{\partial'}{\partial y'} \right\};\;
\frac{\partial}{\partial w}=\frac{\partial'}{\partial w};\;
y\frac{\partial}{\partial y}= (y'+c)\frac{\partial'}{\partial y'}.
\end{equation}
\begin{lemma}
\label{lema:resto} $ \alpha(\xi;t,w,y')\geq
\delta=\delta(\xi;x,w,y)$.
\end{lemma}
\begin{proof} In view of the valuative behavior of the invariant
$\alpha(_{-};x,w)$ and because of the ``monomial'' definition of
$\delta(_{-};x,w)$, it is enough to verify the case that $\xi$ is of
one of the following  monomial types
$$
\xi= y^sx^mw^n x\frac{\partial}{\partial x};\; \xi= y^sx^mw^n
\frac{\partial}{\partial w}; \xi= y^sx^mw^n
y\frac{\partial}{\partial y},
$$
where $\nu(y^sx^m)\geq \delta$. Note that
$$
y^sx^mw^n =x'^{sp+dm}(y'+c)w^n,
$$
where $\nu(x'^{sp+dm})=\nu(y^sx^m)\geq \delta$. Now, in view of the
equations \ref{eq:paqueteney} we have that $\xi=x'^{sp+dm}\xi^{*}$,
where $\alpha(\xi^{*};t,w,y')\geq 0$ and we are done.
\end{proof}
\begin{proposition}
\label{prop:dominantstability} Assume that $\xi$ is completely
prepared with a dominant critical segment and $h\geq 1$. Let us
perform an etale $y$-Puiseux package. After performing finitely many
subsequent etale $w$-Puiseux packages, we obtain ${\mathcal A}'$
such that $\xi$ is completely prepared with respect to ${\mathcal
A}'$ and $h'=h(\xi;{\mathcal A}')\leq \chi$. Moreover, if the
critical polynomial $P_\xi(x,y)$ is a Tchirnhausen polynomial, we
have $h'<\chi\leq h$.
\end{proposition}
\begin{proof}
 Denote
$\xi=\xi_0+\xi^*$. We know that $\delta(\xi^*;x,w,y)>\delta$ and
hence, by Lemma \ref{lema:resto} we have
$\alpha(\xi^*;t,w,y)>\delta$. On the other hand $ \xi_0=\lambda
x^{q}P_\xi(x,y)\frac{\partial}{\partial w}$. After performing the
etale $y$-Puiseux package, we obtain
$$
\xi_0= \lambda x'^{q'}P_\xi(1,y'+c)\frac{\partial'}{\partial w}
$$
where $\nu(t^{q'})=\delta$. If ${\xi^*}'=\lambda^{-1}t^{-q'}\xi^*$,
we have $\alpha({\xi^*}';t,w,y')>0$. Write
\begin{equation}
\label{eq:ypaquetedominante} \xi'=\lambda^{-1}t^{-q'}\xi=
P_\xi(1,y'+c)\frac{\partial'}{\partial w}+{\xi^*}',
\end{equation}
Then $\alpha(\xi';t,w,y')=0$. Let $h'\leq \chi$ be such that
$$
P_\xi(1,y'+c)=y'^{h'}\sum_{s=h'}^\chi \lambda'_sy'^{s-h'};\;
\lambda'_{h'}\ne 0.
$$
It is obvious that $h'\leq\chi\leq h$ and, in view of Remark
\ref{rk:poltchirnhausen}, we have that $h'<\chi\leq h$ in the case
that $P_\xi(x,y)$ is a Tchirnhausen polynomial. Moreover, we see
that $h'=\hbar(\xi';{\mathcal A}')=\hbar(\xi;{\mathcal A}')$. Write
$ \xi=\sum_{s=-1}^\infty y'^s\eta'_s$, as usual, with
$$
\eta'_s=a'_s(t,w)x'\frac{\partial}{\partial'
t}+b'_s(t,w)\frac{\partial'}{\partial
w}+c'_s(t,w)y'\frac{\partial'}{\partial y'}.
$$
Then $\eta'_{h'}$ is $(t,w)$-prepared and dominant in view of
Equation \ref{eq:ypaquetedominante}. By performing new etale
$w$-Puiseux packages to obtain a completely prepared $\xi$, the main
height $h'$ is not modified and we are done.
\end{proof}
\subsection{Stability of the main height. Recessive critical segment}
Take here the situation and notations of the previous Subsection
\ref{seccionestabilidaddominante}.

Let us assume that $\xi$ is completely prepared with respect to
${\mathcal A}$ with a recessive critical segment and $h\geq 1$.
Recall that $\chi\leq h-1$ in view of Remark
\ref{rk:alturaprincipalycritica}. We also have $\xi=\xi_0+\xi^*$
where
$$
\xi_0=\lambda x^{q}{P_\xi(x,y)}\frac{\partial}{\partial y}=\lambda
x^{q}\frac{P_\xi(x,y)}{y}y\frac{\partial}{\partial y},
$$
where
$P_\xi(x,y)=y^{\chi+1}+\sum_{s=-1}^{\chi-1}\lambda_sy^{s+1}x^{q_s}$
is the critical polynomial. After performing an etale $y$-Puiseux
package, we have
$$
\xi_0=\lambda t^{q'}P_\xi(1,y'+c)\frac{\partial'}{\partial y'}
$$
where $\nu(t^{q'})=\delta$. Write
${\xi^*}'=\lambda^{-1}t^{-q'}\xi^*$, as in the proof of Proposition
\ref{prop:dominantstability}.  We have $\alpha({\xi^*}';t,w,y')>0$
and
\begin{equation}
\label{eq:xirecessive}
 \xi'=\frac{1}{\lambda t^{q'}}\xi=
P_\xi(1,y'+c)\frac{\partial}{\partial y'}+{\xi^*}'=\xi'_0+{\xi^*}'.
\end{equation}
Let $-1\leq h'\leq \chi$ be such that
$$
P_\xi(1,y'+c)=y'^{h'+1}\sum_{s=h'}^{\chi}\lambda'_sy'^{s-h'}=y'^{h'+1}Q(y');\quad
Q(0)\ne 0.
$$
It is obvious that $h'\leq\chi\leq h-1$. Moreover, if $P_\xi(x,y)$
is a Tchirnhausen polynomial  we have $h'<\chi\leq h-1$, in view of
Remark \ref{rk:poltchirnhausen}.
\begin{remark}
We have that $h'=\hbar(\xi';t,w,y')$, but the main vertex is not
dominant. For this reason, we will do a coordinate change in the
dependent variables of the type $w''=w+y'$.
\end{remark}
Let us do a coordinate change $w''=w+y'$ to obtain ${\mathcal
A}''=({\mathcal O}',{\mathbf z}''=(t,w'',y'))$. We have
\begin{equation}
\label{eq:cambiorecesivo} \xi'_0=
P_\xi(1,y'+c)\frac{\partial'}{\partial
y'}=y'^{h'+1}Q(y')\left\{\frac{\partial''}{\partial
w''}+\frac{\partial''}{\partial y'}\right\}.
\end{equation}
Let us write $ \xi'=\sum_{s=-1}^\infty y'^s \eta_s''$, where
$$
\eta''_s=a''_s(x',w'')x'\frac{\partial ''}{\partial x'}+
b''_s(x',w'')\frac{\partial ''}{\partial
w''}+c''_s(x',w'')y'\frac{\partial ''}{\partial y'}.
$$
Recalling that $\xi'=\xi'_0+{\xi^*}'$ and
$\alpha({\xi^*}';x',w'',y')>0$, we see from Equations
\ref{eq:cambiorecesivo} that $\alpha(\eta''_{h'+1};x',w'',y')=0$ and
$\eta''_{h'+1}$ is dominant and prepared with respect to
$(t,w'',y')$. In particular $\hbar(\xi';t,w'',y')\leq h'+1$.

As a consequence, by performing new etale $w''$-Puiseux packages, we
obtain $\tilde{\mathcal A}$ such that $\xi'$ is completely prepared
and $h(\xi';\tilde{\mathcal A})\leq h'+1$. Now, recalling that
$h'\leq \chi\leq h-1$ and in the case of a Tchirnhausen critical
polynomial we have $h'< \chi\leq h-1$, we have proved the following
statement:
\begin{proposition}
\label{prop:recessivestability}
 Let $\xi$ be completely prepared with
 a recessive critical segment and assume $h=\hbar(\xi;{\mathcal A})\geq 1$.
 Let us perform an etale $y$-Puiseux package.
After performing a coordinate change in the dependent variables and
finitely many subsequent etale $w$-Puiseux packages, we obtain
$\tilde{\mathcal A}$ such that $\xi$ is completely prepared with
respect to  $\tilde {\mathcal A}$  and $\tilde h=\tilde
\hbar(\xi;\tilde{\mathcal A})\leq \chi+1\leq h$. Moreover, if the
critical polynomial $P_\xi(x,y)$ is a Tchirnhausen polynomial, we
have $\tilde h<\chi+1\leq h$.
\end{proposition}
\subsection{The condition of Tchirnhaus}
\label{subsect:tchirnhaus}
 Let $\xi\in\mbox{Der}_k{\mathcal O}[\log x]$ be completely prepared with respect
  to ${\mathcal A}$. Put $h=h(\xi;{\mathcal A})$, $\chi=\chi(\xi;{\mathcal A})$  and
assume $h\geq 1$.  We have the following possible cases:
\begin{enumerate}
\item[A)] The critical polynomial is not Tchirnhausen and the
critical segment is dominant with $\chi=h$.
\item[B)] The critical polynomial is not Tchirnhausen and the critical
segment is recessive with $\chi=h-1$.
\item[C)] We have one of the following properties:
\begin{enumerate}
\item The critical polynomial is Tchirnhausen. \item The critical segment
is recessive and  $\chi<h-1$. \item The critical segment is dominant
and $\chi<h$.
\end{enumerate}
\end{enumerate}
The last case C corresponds to a winning situation in the sense of
the following proposition
\begin{proposition}
\label{pro:Tchirnahussituation} Assume we are in case C above. Let
us perform an etale $y$-Puiseux package. By performing a subsequent
coordinate change in the dependent variables (if it is necessary)
and finitely many etale $w$-Puiseux packages, we obtain $\tilde
{\mathcal A}$ such that $\xi$ is completely prepared with respect to
$\tilde{\mathcal A}$ and  $\tilde h=h(\tilde {\mathcal
L};\tilde{\mathcal A})< h$.
\end{proposition}
\begin{proof} Direct consequence of Propositions
\ref{prop:dominantstability} and \ref{prop:recessivestability}.
\end{proof}
Next subsections are devoted to the study of situations B and A.

\subsection{Tchirnhausen preparation. Recessive case} In this
subsection we introduce a {\em recessive Tchirnhausen preparation
algorithm} in order to deal with the case B of the preceding
subsection. This algorithm is based on the following two
definitions.
\begin{definition}
\label{def:recessivesigma}
 Let $\varsigma>0$ be a positive element of the
value group $\Gamma\subset {\mathbb R}$. Consider ${\mathcal
A}=({\mathcal O}, (x,w,y))$. We say that $(x,w)$ is {\em recessive} for $\varsigma$
if and only if we have $ \varsigma>\sum_{i=0}^N\nu(w_i)$ for any
finite sequence of etale $w$-Puiseux Packages
$${\mathcal A}={\mathcal A}_0\mapsto {\mathcal A}_1\mapsto \cdots{\mathcal A}_N,$$  where $(x_i,w_i,y)$ is the
coordinate system in ${\mathcal A}_i$, $i=0,1,\ldots, N$.
\end{definition}
 An example of this situation is obtained if we are in the
 case B of Subsection \ref{subsect:tchirnhaus}. More generally, let
 $\xi$ be completely prepared with recessive critical segment and
 put
 $h=\hbar(\xi;x,w,y)$,  $\chi=\chi(\xi;x,w,y)$. Let
  $(\alpha,h)$ be the main vertex and $(\beta,\chi)$ the ``critical''
  vertex. Then $(x,w)$ is  recessive for
 $\varsigma=(h-\chi)\nu(y)-\beta+\alpha$.

 \begin{remark} Assume that $\xi$ and $\mathcal A$ are in the situation of
 case B of Subsection \ref{subsect:tchirnhaus}. Then, there
is an integer number $p\in{\mathbb Z}_{>0}$ such that
$\nu(y)=\nu(x^p)$. Indeed, this is always true when we have a
$\nu$-homogeneous Tchirnhausen polynomial, that we write
$$
P(x,y)=y^m+\lambda_1x^{n_1}y_{m-1}+\cdots+\lambda_mx^{n_m}
$$
since $\lambda_1\ne0$ implies that $\nu(y)=\nu(x^{n_1})$.
\end{remark}

 \begin{definition}
 \label{def:step}
Consider a vector field ${\xi}\in \mbox{Der}_k({\mathcal O})[\log
x]$ and ${\mathcal A}$ with coordinates $(x,w,y)$. Write
$\xi=x^{\tilde q}\xi'$, where $\alpha(\xi';x,w,y)=0$. Consider two
elements $\epsilon$, $\gamma_0$ of the value group $\Gamma$ and take
$h\in {\mathbb Z}_{\geq 1}$. Let us do the decomposition
$\xi'=\sum_{s=-1}^\infty y^s\eta'_s$ associated to $(x,w,y)$. We say
that $(\xi;{\mathcal A};\epsilon, \gamma_0,h,p)$ is a {\em recessive
preparation step of order $p\in{\mathbb Z}_{\geq 0}$} if the
following properties hold
\begin{enumerate}
\item $h=\hbar(\xi;x,w,y)$ and $\gamma_0\leq \nu(y)$.
\item There is $q\in {\mathbb Z}_{\geq 1}$ such that $\epsilon=\nu(x^q)$ and $(x,w)$ is
recessive for  $\gamma_0-\epsilon$ .
\item $\gamma_0\leq \nu(x^{ p})$
\item There are units $U(x,w), V(x,w)$ such that the
levels $\eta'_h$, $\eta'_{h-1}$, and $\eta'_{h-2}$ take the forms
\begin{equation}
\label{eq:nivelesrecesiveprepstep}
\begin{array}{lcl}
\eta'_h&=& U(x,w)\{{\partial}/{\partial w}+x c_h(x,w)
y{\partial}/{\partial y}\},
\\
\eta'_{h-1}  &= & x^{q}V(x,w)\{ xa_{h-1}(x,w)x{\partial}/{\partial
x}+ xb_{h-1}(x,w){\partial}/{\partial w} + y{\partial}/{\partial
y}\},
\\
\eta'_{h-2}&=&a_{h-2}(x,w)x{\partial}/{\partial x}+
b_{h-2}(x,w){\partial}/{\partial w} +x^{q+p}c_{h-2}(x,w)
y{\partial}/{\partial y},
\end{array}
\end{equation}
\end{enumerate}
We say that $(\xi;{\mathcal A};\epsilon, \gamma_0,h,p)$ is a {\em
final recessive step} if in addition we have that
$\nu(y)< \nu(x^p)$.
 \end{definition}

 \begin{remark} Assume that $\xi$ and $\mathcal A$ are in the situation of
 case B of Subsection \ref{subsect:tchirnhaus}. Take $p\in{\mathbb Z}_{>0}$ such that
$\nu(y)=\nu(x^p)$ and $\gamma_0=\nu(y)$. Let $(\alpha,h)$ be the
main vertex and $(\beta,h-1)$ the critical vertex and put
$\epsilon=\beta-\alpha$. Then $(\xi;{\mathcal
A};\epsilon,\gamma_0,h,p)$ is a (non-final) recessive preparation
step or order $p$.
\end{remark}

 \begin{proposition}
\label{prop:recessivepreparationstep}
  Assume that $(\xi,{\mathcal A};\epsilon,
\gamma_0,h,p)$ is a recessive preparation step of order $p$. There
is a coordinate change $y^*=y-x^pg(x,w)$ such that $(\xi,{\mathcal
A}^*;\epsilon,\gamma_0,h,p^*)$ is a recessive preparation step of
order $p^*>p$.
\end{proposition}
\begin{proof}
Take $g(x,w)$ in the Hensel closure of $\mathcal O$ and let us write
$y^*=y-x^pg(x,w)$. Note that $\gamma_0\leq \nu(y^*)$ since
$\gamma_0\leq \nu(y)$ and $\gamma_0\leq \nu(x^pg(x,y))$. The
property that $(x,w)$ is recessive for $\gamma_0-\epsilon$ does not
depend on $y^*$. We have
\begin{eqnarray}
\label{eq:recessprep1} x\frac{\partial}{\partial x}&=&x
\frac{\partial^*}{\partial x}+x^p\left(pg(x,w)+x\frac{\partial
g(x,w)}{\partial
x}\right)\frac{\partial^*}{\partial y^*}\\
\label{eq:recessprep2} \frac{\partial}{\partial w}&=&
\frac{\partial^*}{\partial w}+x^p\frac{\partial g(x,w)}{\partial
w}\frac{\partial^*}{\partial
y^*}\\
\label{eq:recessprep3} y\frac{\partial}{\partial y}&=&
y^*\frac{\partial^*}{\partial y^*}+x^p\frac{\partial^*}{\partial
y^*}
\end{eqnarray}
Let us decompose $\xi'=\sum_{s=-1}^\infty {y^*}^s{\eta'}_s^*$ as
usual with respect to $(x,w,y^*)$. Noting that $q<p$ since
$\epsilon=\nu(x^q)<\gamma_0\leq\nu(x^p)$, it is a straightforward
computation from equations \ref{eq:recessprep1},
\ref{eq:recessprep2} and  \ref{eq:recessprep1} that
$\hbar(\xi';x,w,y^*)=h$ and  ${\eta'}_h^*$, ${\eta'}_{h-1}^*$ and
${\eta'}_{h-2}^*$ take the forms
\begin{equation*}
\begin{array}{lcl}
{\eta'}^*_h&=& U^*(x,w)\{{\partial}/{\partial w}+x c^*_h(x,w)
y{\partial}/{\partial y}\},
\\
{\eta'}^*_{h-1}  &= & x^{q}V^*(x,w)\{
xa^*_{h-1}(x,w)x{\partial}/{\partial x}+
xb^*_{h-1}(x,w){\partial}/{\partial w} + y{\partial}/{\partial y}\},
\\
{\eta'}^*_{h-2}&=&a^*_{h-2}(x,w)x{\partial}/{\partial x}+
b^*_{h-2}(x,w){\partial}/{\partial w} +x^{q+p}c^*_{h-2}(x,w)
y{\partial}/{\partial y}
\end{array}
\end{equation*}
where $U^*(0,0)\ne 0$ and $V^*(0,0)\ne 0$.  In order to end our
proof it is enough to show that $g(x,w)$ may be chosen in such a way
that  $x$ divides $c^*_{h-2}$.Let us put
\begin{equation}
\begin{array}{ccccccc}
F&=&\xi'(y)&=& \sum_{s=0}^\infty y^sF_s(x,w)&=& \sum_{s=0}^\infty
{y^*}^sF^*_s(x,w)\\
H&=&\xi'(x^pg(x,w))&=& \sum_{s=0}^\infty y^sH_s(x,w)&=&
\sum_{s=0}^\infty
{y^*}^sH^*_s(x,w)\\
G&=&\xi'(y^*)&=&F-H&=& \sum_{s=0}^\infty {y^*}^sG^*_s(x,w)
\end{array}
\end{equation}
We have to prove that $G^*_{h-1}(x,w)$ is divisible by $x^{q+p+1}$
after a suitable choice of $g(x,w)$. Let us decompose
\begin{equation}
\begin{array}{cccc}
F&=&\tilde F + \overline F&; \tilde F= y^{h-1}F_{h-1}+ y^{h}F_h\\
H&=&\tilde H + \overline H&; \tilde H= y^{h-1}H_{h-1}+ y^{h}H_h
\end{array}
\end{equation}
We have that ${\overline F}^*_{h-1}$ and ${\overline H}^*_{h-1}$ are
divisible by $x^{q+p+1}$, since they are is divisible by $x^{2p}$
and $2p\geq p+q+1$. Note also that
$$
\tilde H= J+K;\quad J=y^{h-1}\eta'_{h-1}(x^pg(x,w)),\;  K=
y^{h}\eta'_{h}(x^pg(x,w)).
$$
Moreover, $J$ is divisible by $x^{q+p+1}$ in view of form of
$\eta'_{h-1}$ in Definition  \ref{def:step}. We also have that
$x^{2p}$ divides $K^*_{h-1}$ and $2p>p+q+1$.

 Thus,
we have only to prove that after a suitable choice of $g(x,w)$ we
can obtain that $\tilde F^*_{h-1}$ is divisible by $x^{q+p+1}$.
Recall that $\tilde F=y^{h-1}(yF_{h}+F_{h-1})$ where
\begin{equation}
\begin{array}{c}
yF_h=\eta'_{h-1}(y)=yV(x,w)x^q\\
F_{h-1}=\eta'_{h-2}(y)/y=x^{q+p}c_{h-2}(x,w).
\end{array}
\end{equation}
Now, write $c_{h-1}(x,w)=f_0(w)+xf_1(x,w)$. If we put
$g(x,w)=-f_0(w)/V(x,w)$ we are done.
\end{proof}

 Let us show how to obtain a final recessive step.
 We start with $\xi$ completely prepared with respect to $\mathcal A$ in the case B of
 Subsection \ref{subsect:tchirnhaus}.  Thus we have a recessive
 preparation step $(\xi,{\mathcal A};\epsilon,\gamma_0,h,p)$ of
 order $p$, where $\gamma_0=\nu(y)=\nu(x^p)$. Since
 $\nu(x^p)=\nu(y)$, it is not a final recessive step. We do a
 coordinate change $y_1=y-x^pg_1(x,w)$ as
in Proposition \ref{prop:recessivepreparationstep} to obtain a new
recessive preparation step $(\xi,{\mathcal
A}_1;\epsilon,\gamma_0,h,p_1)$ with $p_1>p$. We repeat to obtain
 $$
y_{j+1}=y_j-x^{p_j}g_j(x,w),
 $$
 where $(x,w,y_j)$ are the coordinates for  a recessive preparation step
$(\xi;{\mathcal A}_j;\epsilon,\gamma_0,h,p_j)$
  of order $p_j$
 and $p_{j+1}>p_j$. There are two possibilities:
 \begin{enumerate}
 \item
  We have $\nu(y_j)\geq
 \nu(x^{p_j})$ for all $j$. In this case we obtain a transversal formal maximal
 contact element $\hat f\in\hat{\mathcal O}$ as the limit of the $y_j$.
 \item There is an index $j_0$ such that
 $\nu(y_{j_0})<\nu(x^{p_{j_0}})$. In this case we obtain a
  final recessive step $(\xi;{\mathcal A}_{j_0};\epsilon,\gamma_0,h,p_{j_0})$.
 \end{enumerate}
\begin{proposition}
\label{prop:caseB}
 Assume that we have a final recessive step
$(\xi;{\mathcal A};\epsilon,\gamma_0,h,p)$.
 After performing finitely many
$w$-Puiseux packages we obtain ${\mathcal A}'$ such that $\xi$ is
completely prepared with respect to ${\mathcal A}'$ and  we are in
the winning situation C of Subsection \ref{subsect:tchirnhaus}.
\end{proposition}
\begin{proof} Note that if $U(x,w,y)$ is a unit, then
 $(U(x,w,y)\xi;{\mathcal A};\epsilon,\gamma_0,h,p)$
is still a final recessive step.

Let us perform an etale $w$-Puiseux package to obtain ${\mathcal
A}_1$ whose coordinates are $(t,w_1,y)$, where $t^{d}=x$ and
$w_1=w/t^{\tilde p}-c$. In view of Equations
\ref{eq:nivelesrecesiveprepstep} we see that $\xi$ is main-vertex
prepared with respect to $(x,w,y)$ and hence the main height $h$ is
not changed under the etale $w$-Puiseux package. The form of
Equations \ref{eq:nivelesrecesiveprepstep} persists, with the
following observations:
\begin{enumerate}
\item The parameter $\epsilon$ is transformed into $\epsilon_1=\epsilon
+\nu(w)$. Anyway, we still have that $(t,w_1)$ is recessive for
$\gamma_0-\epsilon_1$ (see Definition \ref{def:recessivesigma}).
\item The order $p$ is transformed into $p_1=pd$.
\end{enumerate}
In particular $(\xi;{\mathcal A}_1;\epsilon_1,\gamma_0,h,p_1)$ is a
final recessive preparation step.

Thus, we can multiply by a unit $\xi$ an do successive etale
$w$-Puiseux packages in order to obtain that in addition $\xi$ is
completely prepared with respect to ${\mathcal A}$. Let us look at
Equations \ref{eq:nivelesrecesiveprepstep}. Let us put
$$
\alpha_{h-1}=\alpha(\eta'_{h-1};x,w)\;\quad
\alpha_{h-2}=\alpha(\eta'_{h-2};x,w).
$$
From the form of $\eta'_{h-1}$ in Equations
\ref{eq:nivelesrecesiveprepstep} we have that
$\alpha_{h-1}=\nu(x^q)=\epsilon$. Since
$\epsilon<\gamma_0\leq\nu(y)$ we see that $\chi<h$. In particular,
if we are in the case of a dominant critical segment, we are in one
of the winning situations C. Assume that $\chi=h-1$ and we have a
recessive critical segment. Recall that
$$
\eta'_{h-2}=a_{h-2}(x,w)x{\partial}/{\partial x}+
b_{h-2}(x,w){\partial}/{\partial w} +x^{q+p}c_{h-2}(x,w)
y{\partial}/{\partial y}
$$
and we are assuming moreover that $\eta'_{h-2}$ is prepared. If this
level $h-2$ is dominant we are in a winning situation C, since it
cannot be in the critical segment and thus the critical polynomial
is a Tchirnhausen polynomial.  If the level $h-2$ is recessive, from
the above form of $\eta'_{h-1}$ we deduce that
$\alpha_{h-2}=\nu(x^{q+p})=\epsilon+\nu(x^p)$. But we know that
$\nu(x^p)>\nu(y)$ and thus the level $h-2$ cannot be in the critical
segment. This ends the proof.
\end{proof}
\subsection{Tchirnhausen preparation. Dominant case}
\label{subsec:thchirndominant} In this Subsection we assume we are
in case A of Subsection \ref{subsect:tchirnhaus}.  That is, we have
$\xi$ completely prepared with respect to ${\mathcal A}$, the
critical segment is dominant with $\chi=h$ and the critical
polynomial is not Tchirnhausen. We also assume that $h\geq 2$ since
the cases $h\leq 1$ correspond to log-elementary singularities.
\begin{proposition} We can perform a coordinate change
$y^*=y-x^pg(x,w)$, with $\nu(x^p)=\nu(y)$ to obtain ${\mathcal A}^*$
is such a way that after performing finitely many etale $w$-Puiseux
packages, we get ${\mathcal A}'$ such that $\xi$ is completely
prepared with respect to ${\mathcal A}'$ with $h=\hbar(\xi;{\mathcal
A}')$ and we are in one of the situations B or C of
\ref{subsect:tchirnhaus}.
\end{proposition}
\begin{proof} Since the critical polynomial is not Tchirnhaus, we
have that $\nu(y)=\nu(x^p)$ for some $p\in{\mathbb Z}_{\geq 0}$. Up
to multiply $\xi$ by a power of $x$, let us assume without loss of
generality that $\alpha(\xi;x,w,y)=0$. Denote
$F=\xi(w)=\sum_{s=0}^\infty y^{s} F_s(x,w)$. We know that
$F_{h}(0,0)\ne 0$. Moreover, in view of our hypothesis, we have
$F_{h-1}(x,w)=x^pG_{h-1}(x,w)$, where $G_{h-1}(0,0)\ne 0$. By an
argument like in Proposition \ref{prop:recessivepreparationstep}, we
can find a coordinate change of the form $y^*=y-x^pg(x,w)$ such that
$$
F^*(x,w,y^*)=F(x,w,y^*+x^pg(x,w))=\sum_{s=0}^\infty{y^*}^s
F^*_s(x,w)
$$
satisfies that $F^*_{h-1}=0$. This condition eliminates the level
$h-1$ from the critical segment if we persist in the situation A
after subsequent etale $w$-Puiseux packages.
\end{proof}
\section{Maximal contact}
In this section we prove Theorem \ref{teo:cuatro}. Recall that we
consider the case when $n=3$ and $\nu$ is a valuation of arquimedean
rank one with $\kappa_\nu=k$. We have a projective model $M_0$ of
$K$, where  $P_0$ is the center of $\nu$ at $M_0$. We assume that
$P_0$ is a regular point of $M_0$ and there is $\hat
f\in\widehat{\mathcal O}_{M_0,P_0}$ that has transversal maximal
contact with $\nu$. The rational rank $r$ can be supposed to be
$r=1$ or $r=2$, since if $r=3$ the definition of transversal maximal
contact makes no-sense (see the Introduction).

The computations in this section are essentially contained in the
paper \cite{Can-M-R}, but we include them for the sake of
completeness.

\subsection{Maximal contact with rational rank two.}
\label{subsec:maxconrangodos}
 Take a regular system of parameters
$(x_1,x_2,y)$ of ${\mathcal O}_{M_0,P_0}$ such that
$\nu(x_1),\nu(x_2)$ are ${\mathbb Z}$-linearly independent and
$$
\hat f=y+\sum_{i,j}\lambda_{ij}x_1^ix_2^j.
$$
Since $\nu$ is arquimedian, we may write $\hat f$ as the Krull limit
$\hat f=\lim_{\mu\rightarrow \infty} { f}_{\mu} $, where
$$ { f}_{\mu}=y+\sum_{i\nu(x_1)+j\nu(x_2)\leq
\mu}\lambda_{ij}x_1^ix_2^j\in{\mathcal O}_{M_0,P_0}.
$$
Note that $\nu(f_\mu)>\mu$ and, more precisely we have
$$\nu({f}_{\mu})=\min\{\nu(x_1^ix_2^j);\lambda_{ij}\ne 0, \,
\nu(x_1^ix_2^j)>\mu \}.$$
  In this paragraph we denote
$Y_0=\{P_0\}$, $Y_1=\{x_1={\hat f}=0\}$ and $Y_2=\{x_2={\hat
f}=0\}$.

The next Lemma \ref{lema:catorce} may be proved by standard
computations in
 terms of blow-ups and valuations and we leave the verification to the reader:
\begin{lemma}
\label{lema:catorce} Let $\pi:M'\rightarrow M_0$ be the blow-up of
$M_0$ with one of the centers $Y_0$, $Y_1$ or $Y_2$ and assume that
if we use $Y_1$, respectively $Y_2$, as a center, then ${\hat
f}(0,x_2,y)\in{\mathcal O}_{M_0,P_0}$, respectively  ${\hat
f}(x_1,0,y)\in{\mathcal O}_{M_0,P_0}$. Let $P'\in M'$ be the center
of $\nu$ at $M'$. Then $P'$ belongs to the strict transform of
${\hat f}=0$. More precisely , we have
 the following cases:
 \begin{description}
\item[T-01] The center is $Y_0$ and $\mu=\nu(x_1)<\nu(x_2)$. In this case $P_1$ is
in the strict transform of the formal curve $x_2={\hat f}=0$ and
there is a regular system of parameters $(x'_1,x'_2,{y^*})$ at
${\mathcal O}_{M',P'}$ such that $x'_1=x_1$, $x'_2=x_2/x_1$,
${y^*}={\hat f}_\mu/x_1$. Moreover
$${\hat f}'={\hat f}/x_1={y^*}+\sum_{i\nu(x_1)+j\nu(x_2)>
\mu}\lambda_{ij}{x'}_1^{i+j-1}{x'}_2^j\in{\mathcal O}_{M',P'}$$ has
transversal maximal contact with $\nu$.
\item[T-02] The center is $Y_0$ and $\mu=\nu(x_2)<\nu(x_1)$. Similar
to T-01.
\item[T-1] The center is $Y_1$, where $\mu=\nu(x_1)$. In this case $P_1$ is
 the only point over $P_0$ in the strict transform of  ${\hat f}=0$ and there is a regular system of parameters
$(x'_1,x'_2,y^*)$ at ${\mathcal O}_{M',P'}$ such that $x'_1=x_1$,
$x'_2=x_2$, $y^*=({f}_\mu+ {\hat h}_\mu(0,x_2))/x_1$, where ${\hat
h}_\mu(x_1,x_2)={\hat f}-f_\mu$. Moreover
$${\hat f}'={\hat f}/x_1=y^*+\sum_{\nu(x_1^ix_2^j)>
\mu;\, i\geq 1}\lambda_{ij}{x'}_1^{i-1}{x'}_2^j\in{\mathcal
O}_{M',P'}$$ has transversal maximal contact with $\nu$.
\item[T-2] The center is $Y_2$, where $\mu=\nu(x_2)$. Similar to
T-1.
\end{description}
\end{lemma}
Take a generator $\xi_0$ of ${\mathcal L}_{M_0,P_0}[\log x_1x_2]$.
Define the formal vector field $\hat\xi$ to be $\hat\xi=\xi_0$ if
$\hat f$ divides $\xi_0(\hat f)$ (this corresponds to saying that
$\hat f=0$ defines a formal invariant hypersurface) and
$\hat\xi=\hat f\xi_0$ if $\hat f$ does not divide $\xi_0(\hat f)$.
Let us write
$$
\hat\xi=\hat a_1x_1\frac{\partial}{\partial x_1}+ \hat
a_2x_2\frac{\partial}{\partial x_2}+\hat b\hat
f\frac{\partial}{\partial \hat f}.
$$
Note that $\hat a_1,\hat a_2, \hat b$ have no common factors. The
{\em adapted (or logarithmic) order \/} of ${\mathcal L}$ at $P_0$
with respect to $x_1x_2\hat f$ is
$$
\mbox{LogOrd}({\mathcal L}, {\mathcal O}_{M_0,P_0}; x_1x_2\hat
f)=\mbox{ord}_{\widehat{\mathcal M}_{M_0,P_0}}(\hat a_1,\hat
a_2,\hat b) \in {\mathbb Z}_{\geq 0},
$$
where $\mbox{ord}_{\widehat{\mathcal M}_{M_0,P_0}}(-)$ means the
$\widehat{\mathcal M}_{M_0,P_0}$-adic order (see also \cite{Can}).

Put $\zeta=\mbox{LogOrd}({\mathcal L}, {\mathcal O}_{M_0,P_0};
x_1x_2\hat f)$. We say that $Y_1$ is {\em permissible\/} for
${\mathcal L}$ adapted to $x_1x_2$ if the two following properties
hold:
\begin{enumerate}
\item ${\hat f}(0,x_2,y)\in{\mathcal
O}_{M_0,P_0}$. (Hence $Y_1$ is a subvariety of $M_0$)
\item $\mbox{ord}_{(x_1,\hat f)}(\hat a_1,\hat
a_2,{\hat b})= \zeta$.
\end{enumerate}
We give a symmetric definition for  $Y_2$ being permissible. By
definition $Y_0$ is always permissible.

\begin{remark}
\label{observacioncinco}
 If
 $\zeta\geq 2$, the condition 2 above implies condition 1, since in this case, the curve
$Y_1$ must be contained in the locus
$$\xi_0(x_1)/x_1=\xi_0(x_1)/x_2=\xi_0(y)=0.$$
If $\zeta=1$ and $\hat \xi=\xi_0$, the same argument holds.
\end{remark}
\begin{lemma}
\label{lemaquince} Let $\pi:M'\rightarrow M_0$ be the blow-up of
$M_0$ with a permissible center $Y_0$, $Y_1$ or $Y_2$. Let $P'\in
M'$ be the center of $\nu$ at $M'$. Then
$$
\mbox{\rm LogOrd}({\mathcal L}, {\mathcal O}_{M_0,P_0}; x_1x_2\hat
f)\geq \mbox{\rm LogOrd}({\mathcal L}, {\mathcal O}_{M',P'};
x'_1x'_2{\hat f}').
$$
\end{lemma}
\begin{proof} We may assume that either the center of
the blow-up is $Y_1$ or it is $Y_0$ and $\nu(x_1)<\nu(x_2)$ (the
other cases follow from these by interchanging the roles of
$x_1,x_2$). Then we have ${\hat f}'={\hat f}/x_1$ and
$\hat\xi'=x_1^{-\zeta}\hat\xi$ where
$$
\hat a'_1=x_1^{-\zeta}\hat a_1;\; {\hat b}'=x_1^{-\zeta}({\hat
b}-\hat a'_1),
$$
and $\hat a'_2=x_1^{-\zeta}(\hat a_2-\hat a_1)$ if $Y_0$, $\hat
a'_2=x_1^{-\zeta}\hat a_2$ if $Y_1$. The rest of the proof is given
by the standard results on the blow-up of equimultiple centers.
\end{proof}
 We proceed by induction on $\zeta$.  First, consider
   the case $\zeta\geq 2$.

We now define Hironaka's characteristic polygons (see for instance
\cite{Cos-G-O}).  Take an element $\hat g=\sum_s{\hat
f}^sg_{ijs}x_1^ix_2^j\in\widehat {\mathcal
O}_{M_0,P_0}={k}[[x_1,x_2,\hat f]]$ and an integer $\eta\in {\mathbb
Z}_{> 0}$. The Hironaka's {\em characteristic polygon\/}
$\Delta(\hat g; x_1,x_2,\hat f;\eta)$ is the positive convex hull in
${\mathbb R}_{\geq 0}^2$ of the points of the form
$(i/(\eta-s),j/(\eta-s))$, where $g_{ijs}\ne 0$ and $s<\eta$. Given
a list $\{\hat g_l\}$ we define $\Delta(\{\hat g_l\}; x_1,x_2,\hat
f;\eta)$ to be the convex hull of the union of the $\Delta(\hat g_l;
x_1,x_2,\hat f;\eta)$. Now, we define
$$
\Delta({\mathcal L}; x_1,x_2,\hat f;\eta)=\Delta(\{\hat a_1,\hat
a_2,\hat b\}; x_1,x_2,\hat f;\eta).
$$
Let us list the properties of $\Delta_\eta=\Delta({\mathcal L};
x_1,x_2,\hat f;\eta)$, similar to those used by Hironaka in his
Bowdoin College Memoir \cite{Cos-G-O}:
\begin{enumerate}
\item $\Delta_\eta\ne
\emptyset$. Otherwise $\hat a_1,\hat a_2, \hat b$ would be divisible
by $\hat f$.
\item $\Delta_\eta\subset\{(u,v);\; u+v\geq 1\}$
iff  $\zeta\geq \eta$.
\item $\Delta_\zeta\subset\{(u,v);\; u\geq
1\}$ iff condition 2 of permissibility holds for $Y_1$.
\item $\Delta_\zeta\subset\{(u,v);\; v\geq
1\}$ iff condition 2 of permissibility holds for $Y_2$.
\end{enumerate}
The characteristic polygon behaves under blow-up as in the classical
case of varieties, as we show in the next Lemma \ref{lema:dieciseis}.
To see this, let us introduce the linear mappings
$\sigma_{01},\sigma_{02},\sigma_1,\sigma_2$ defined as follows
$$\begin{array}{cc}
\sigma_{01}(u,v)=(u+v-1,v),&
\sigma_1(u,v)=(u-1,v),\\
\sigma_{02}(u,v)=(u,u+v-1),& \sigma_2(u,v)=(u,v-1).
\end{array}
$$
\begin{lemma}
\label{lema:dieciseis} Keep notations as in Lemma
\ref{lema:catorce}. Let $\pi:M'\rightarrow M_0$ be the blow-up of
$M_0$ with a permissible center $Y_0$, $Y_1$ or $Y_2$. Let $P'\in
M'$ be the center of $\nu$ at $M'$. Put $\Delta=\Delta({\mathcal L};
(x_1,x_2,{\hat f});\zeta)$. Then the characteristic polygon
$\Delta({\mathcal L}; (x'_1,x'_2,{\hat f}');\zeta)$ is the positive
convex hull of
$$\sigma_{01}(\Delta),\sigma_{02}(\Delta),\sigma_1(\Delta),\sigma_2(\Delta)$$
if we are respectively in the cases \textbf{T-01}, \textbf{T-02}, \textbf{T-1} and \textbf{T-2} of
Lemma \ref{lema:catorce}.
\end{lemma}
\begin{proof} Let $I\subset \widehat{\mathcal O}_{M_0,P_0}$ be the
ideal generated by $\hat a_1,\hat a_2,\hat f$. Then the ideal
$I'\subset \widehat{\mathcal O}'_{M_0,P_0}$ generated by $\hat
a_1',\hat a_2',\hat f'$ is $I'=x_1^{-\zeta}I$, respectively
$I'=y_1^{-\zeta}I$ if we are in the cases $(01), (1)$, respectively
$(02), (2)$. Now we apply the classical remarks of Hironaka in his
Bowdoin College seminar \cite{Cos-G-O}.
\end{proof}
Now, we choose the following strategy to blow up. We select the
blow-up center  $Y_0$ until the characteristic polygon has only one
vertex, this occurs after finitely many steps. Then, since we are in
the case $\zeta\geq 2$, at least one of the centers $Y_1,Y_2$ is
permissible, since it is equimultiple. Blow-up this curve. After
finitely many operations the characteristic polygon intersects
$\{(u,v);\; u+v<1\}$ and hence the logarithmic order drops. We
arrive in this way to the case $\zeta\leq 1$.

Assume now that $\zeta\leq 1$. If $\zeta=0$ and $\hat\xi=\xi_0$, we
get an elementary singularity and if $\hat\xi=\hat f\xi_0 $ the
foliation is in fact non-singular. Assume that $\zeta=1$. By Remark
\ref{observacioncinco}, the case $\hat\xi=\xi_0$ can be handled as
before. So we consider only the case $\hat\xi=\hat f\xi_0$.
Blowing-up the origin (that is, we take the center $Y_0$ each time),
we get as above that the characteristic polygon has exactly one
vertex of integer coordinates, say $(\alpha,\beta)\in{\mathbb
Z}_{\geq 0}^2$, where $\alpha+\beta\geq 1$. Assume that
$\alpha+\beta\geq 2$; since $\zeta=1$, we have
 either
$\mbox{ord}_{x_1,x_2,{\hat f}}(\xi_0(x_1),\xi_0(x_2))=0$  or $\hat
b(0,0,\hat f)={\hat f}U(\hat f)$, with $U(0)\ne 0$. In both cases
$\xi_0$ is non-nilpotent and we obtain an elementary singularity. It
remains to study the case $\alpha+\beta=1$. We have two
possibilities $(\alpha,\beta)=(1,0)$ and $(\alpha,\beta)=(0,1)$,
that can be treated in a similar way. Consider, for instance, the
case $(\alpha,\beta)=(1,0)$, if $\nu(x_1)>\nu(x_2)$, we are done by
blowing-up the origin, since we get $\zeta=0$; if
$\nu(x_1)<\nu(x_2)$ the situation repeats itself, but this cannot
occur infinitely many times, since we are dealing with an
arquimedian valuation $\nu$.

This ends the proof of Theorem \ref{teo:cuatro} in the case of
rational rank $r=2$.
\subsection{Maximal contact with rational rank one}
 Take a regular system of
parameters $(x,w,y)$ of ${\mathcal O}_{M_0,P_0}$ such that
$$
\hat f=y+\sum_{i,j}\lambda_{ij}x^iw^j.
$$
In this paragraph we denote $Y_0=\{P_0\}$ and $Y_1=\{x_1={\hat
f}=0\}$. The next Lemma \ref{lema:diecisiete}  may be proved by
standard computations in terms of blow-ups and valuations and we
left the verification to the reader:
\begin{lemma}
\label{lema:diecisiete} Let $\pi:M'\rightarrow M_0$ be the blow-up
of $M_0$ with one of the centers $Y_0$ or $Y_1$
 and assume that if we use $Y_1$ as a center, then ${\hat f}(0,w,y)\in{\mathcal
O}_{M_0,P_0}$. Then, the center $P'\in M'$  of $\nu$ at $M'$ belongs
to the strict transform of ${\hat f}=0$. More precisely,  we have
the following cases:
\begin{description}
\item[$T_{01}$] The center is $Y_0$ and $\mu=\nu(x)<\nu(w)$. In this case $P_1$ is
in the strict transform of the formal curve $w={\hat f}=0$ and there
is a regular system of parameters $(x',w',{y^*})$ of ${\mathcal
O}_{M',P'}$ such that $x'=x$, $w'=w/x$, ${\hat f}'={\hat f}/x$ has
transversal maximal contact with $\nu$ and has the form
$$
{\hat f}'=y^*+\sum_{i,j}\lambda'_{ij}x'^iw'^j.
$$
\item[$T_{02}$] The center is $Y_0$ and $\mu=\nu(w)<\nu(x)$. Similar to
the previous case, by the roles of $x,w$ interchanged.
\item[$T_{01},c$] The center is $Y_0$ and $\mu=\nu(x)=\nu(w)$. Take a parameter $c\in k$
such that $\nu(w-cx)>\nu(x)$. We do the coordinate change $w^*=w-cx$
and we proceed as in the case $(01)$.
\item[$T_1$] The center is $Y_1$. In this case $P_1$ is
in the strict transform of ${\hat f}=0$ and there is a regular
system of parameters $(x',w',{y^*})$ of ${\mathcal O}_{M',P'}$ such
that $x'=x$, $w'=w$, ${\hat f}'={\hat f}/x$ has transversal maximal
contact with $\nu$ and it is written as
$$
{\hat f}'=y^*+\sum_{i,j}\lambda'_{ij}x'^iw'^j.
$$
\end{description}
\end{lemma}
We define an $\{x,w,y, {\hat f}\}$-{\em formal Puiseux package} to
be a sequence of blow-ups
$$
M_0\leftarrow M_1\leftarrow\cdots\leftarrow M_N= M'
$$
such that: \begin{enumerate}
\item Each blow-up has center at the
center $P_i\in M_i$ of the valuation in the projective model $M_i$.
\item We get $(x_i,w_i,y_i,{\hat f}_i)$ at each $P_i$, obtained as in
Lemma \ref{lema:diecisiete}, starting  from $(x_0,w_0,y_0, {\hat
f}_0)=(x,w,y, {\hat f})$.
\item Each blow-up is given by $T_{01}$ or $T_{02}$, except  the
last blow-up that is given by $(T_{01},c)$, with $c\ne0$.
\end{enumerate}
A $\{x,w,y, {\hat f}\}$-{formal Puiseux package}  exists and is
unique. More precisely, if we put $\nu(w^d/x^p)=0$ and $c$ is such
that $\nu(w^d/x^p-c)>0$, the sequence of blow-ups is the reduction
of singularities of the formal curve $w^d-cx^p=\hat f=0$.

Now, let us consider a generator $\xi_0$ of ${\mathcal
L}_{M_0,P_0}[\log x]$, that we write a follows:
$$
\xi_0=a(x,w,\hat f)x\frac{\partial}{\partial x}+ b(x,w,\hat
f)\frac{\partial}{\partial w}+\hat h(x,w,\hat
f)\frac{\partial}{\partial \hat f},
$$
where $a=\xi_0(x)/x$, $b=\xi_0(w)$, $\hat h={\xi_0}(\hat f)$. Note
that $a,b$ and $\hat h$ have no common factors. There are two cases
that we will consider separately
\begin{enumerate}
\item The formal hypersurface $\hat f=0$ is invariant by $\xi_0$.
Then $\hat f$ divides $\hat h$ and we can put $\hat h={\hat g}\hat
f$.
\item The formal hypersurface $\hat f=0$ is not invariant by $\xi_0$.
Then $\hat f$ does not divide $\hat h$ and thus  ${\hat f}a, {\hat
f}b,\hat h$ have no common factors.
\end{enumerate}
Let us put $\hat \xi_0=\hat \xi_0$ if $\hat f=0$ is invariant and
$\hat \xi_0=\hat f \xi_0$ if $\hat f$ is not invariant. In both
cases we denote
$$
\hat a_0={\hat \xi}_0(x)/x;\; \hat b_0={\hat \xi}_0(w); \hat
g_0={\hat \xi}_0(\hat f)/\hat f.
$$
Then $\hat a_0,\hat b_0,\hat g_0$ have no common factors. Define the
{\em logarithmic order} as $$ \mbox{LogOrd}({\mathcal L};x{\hat
f})=\mbox{ord}_{\widehat {\mathcal M}_0}(\hat a_0,\hat b_0,\hat
g_0).$$
\begin{lemma}
\label{dieciocho} Let $\pi: M'\rightarrow M_0$ be given by the
$\{x,w,y,\hat f\}$-formal Puiseux package and let $(x', w', y',{\hat
f}')$ be the resulting list at the center $ P'$ of $\nu$ in $M'$.
Then
$$
\mbox{\rm LogOrd}({\mathcal L}; x'{\hat f}')\leq\mbox{\rm
LogOrd}({\mathcal L};x{\hat f}).
$$
\end{lemma}
\begin{proof} The result is true under each of the blow-ups of the
sequence given by the $\{x,w,y,\hat f\}$-formal Puiseux package.
This is a  standard verification which is also a part of the proof
of the vertical stability of the adapted order given in
\cite{Can}.\end{proof}
 Consider an expansion $\hat \xi=\sum_{s\geq 0} {\hat f}^s \hat \eta_s(x,w)
 $, where
 $$\hat \eta_s(x,w)=\hat a_s(x,w)x\frac{\partial}{\partial x}+
 \hat b_s(x,w)\frac{\partial}{\partial w}+
 \hat g_s(x,w)\hat f\frac{\partial}{\partial \hat f}.
 $$
We say that $\hat\eta_s$ is {\em formally strongly prepared\/} if we
can write
\begin{equation}
 \hat\eta_s=
 x^\rho \hat U(x,w)\theta+x^\tau \hat V(x,w)\hat f\frac{\partial}{\partial \hat f};\quad
 \theta=x\hat h(x,w)x\frac{\partial}{\partial x}+\frac{\partial}{\partial w}
\end{equation}
satisfying the same properties as in Definition
\ref{defstrongprepared}, that is
\begin{enumerate}
\item $\rho,\tau\in {\mathbb Z}\cup\{+\infty\}$, with $\rho\ne \tau$.
\item $\hat U=\lambda+x(\cdots)$ and $\hat V=\mu+x(\cdots)$, where
$\lambda,\mu\in k\setminus\{0\}$. (Except if $\rho=+\infty$ or
$\tau=\infty$, that indicates that $\hat U$, respectively $\hat V$
is identically zero)
\end{enumerate}
By the same proof as in \ref{pro:diez}, we have
\begin{proposition}
\label{pro:formalstrongpreparation} Assume that $\hat \eta_s\ne 0$,
then after finitely many formal Puiseux packages we obtain
$\hat\eta_s$ that is formally strongly prepared.
\end{proposition}

 Let us work by induction on
$\varrho=\mbox{LogOrd}({\mathcal L};x{\hat
 f})$. If $\varrho\leq 1$ we have a log-elementary singularity.
 Assume that $\varrho\geq 2$. By Proposition \ref{pro:formalstrongpreparation}, after finitely many
 formal Puiseux packages, the vector field  $\hat
 \xi$ can be written as
$ \hat \xi=\sum_{0\leq s} {\hat f}^s \hat \eta'_s(x,w)
 $, where
 $$\hat\eta_s=
 x^{\rho_s} \hat U_s(x,w)\theta_s+x^{\tau_s} \hat V_s(x,w)\hat f\frac{\partial}{\partial \hat f};\quad
 \theta_s=x\hat h_s(x,w)x\frac{\partial}{\partial x}+\frac{\partial}{\partial
 w}$$
is formally strongly prepared for any $s\leq \varrho$. Let us put
$m_s=\min\{\rho_s,\tau_s\}$.  let us also define
$$
\delta=\min \left\{\frac{m_s}{\varrho-s};\; s<\varrho\right\}.
$$
It is clear that $1\leq \delta<\infty$, since the adapted order is
$\varrho$.

 Consider the ideal $(x,\hat f)$.  Since
$\varrho\geq 2$, this ideal gives a curve in the singular locus of
$\mathcal L$. Thus we can blow-up it. After blowing-up, we get that
$\varrho'\leq \varrho$ and $\delta'=\delta-1$ if $\varrho'=\varrho$.
This ends the proof of Theorem \ref{teo:cuatro}.
\part{Higher rank and higher dimensional valuations}
In this part we complete the proof of Theorem \ref{teouno} by
considering valuations of higher arquimedean rank or of dimension
bigger than zero. In fact these cases correspond to situations
simpler than in Part 1, since they are ``essentially'' of ambient
dimension two.
\section{Higher rank valuations}
In this section we assume that $n=3$ and  $\kappa_\nu=k$ but $\nu$
has rank bigger than one, that is, the value group $\Gamma$ is not
arquimedean. If the rational rank $r=3$, there is no difference with
the computations in the case of an arquimedean valuation done in
Section \ref{secciondos}. The only remaining situation is $r=2$. Let
us consider this situation.

We can work in terms of parameterized regular local models
${\mathcal A}=({\mathcal O},{\mathbf z}=({\mathbf x},y))$ as in the
case of a real valuation of rational rank two (Sections 3-4). Let us
consider the following statement
\begin{quote}
{\bf TRI: Trivial ramification index assumption: } After performing
any finite sequence of $y$-Puiseux packages, coordinate blow-ups in
the independent variables and coordinate changes in the dependent
variable, we obtain ${\mathcal A}=({\mathcal O},{\mathbf
z}=({\mathbf x},y))$ such that the ramification index is equal to
one. That is $\nu(y)=\nu(x_1^{p_1}x_2^{ p_2})$ for
$(p_1,p_2)\in{\mathbb Z}^2$.
\end{quote}
Following the same arguments as in Sections 3-4 we obtain
\begin{proposition} Assume that the {\em Trivial Ramification
Index Assumption does not hold} after performing any finite sequence
of  $y$-Puiseux packages, coordinate blow-ups in the independent
variables and coordinate changes in the dependent variable. Then we
can obtain a log-elementary ${\mathcal L}_{\mathcal A}$ after
performing such a finite sequence of transformations.
\end{proposition}
Thus, we assume that TRI holds. We can work by induction on the main
height $h=\hbar(\xi;{\mathbf x},y)$ of a generator of ${\mathcal
L}$. By the same arguments in Sections 3-4, if the critical
polynomial is not Thchirnhaus, we can win. So, we find an element at
the level $h-1$ corresponding to the critical polynomial. This means
that $(p_1,p_2)\in{\mathbf Z}_{\geq 0}^2$, since this point of the
support is associated to a monomial $x_1^{p_1}x_2^{p_2}$ appearing
in the coefficients of $\xi$. Now, we can do the coordinate change
$$
y_1=y-c_1x_1^{p_1}x_2^{p_2};\quad \nu(y_1)>\nu(y).
$$
The situation repeats. We obtain a formal element $\hat f=y-\sum
c_ix_1^{p_{i1}}x_2^{p_{i2}}$. Now we can apply to $\hat f$ the same
arguments as in Subsection \ref{subsec:maxconrangodos}.
\section{Higher dimensional valuations}
In this section we assume that $n=3$ and  $\kappa_\nu\ne k$. We look
for
 a projective model $M$ of $K$ and a birational morphism
$M\rightarrow M_0$ such that the center $Y$ of $\nu$ at $M$ has
dimension $\geq 1$ and a generic point of $Y$ is a regular point of
$M$ which is log-elementary for $\mathcal L$. Since $k$ is
algebraically closed, the assumption $\kappa_\nu\ne k$ implies that
$\dim\nu\geq 1$, where $\dim\nu$ is the transcendence degree of
$\kappa_\nu/k$. Applying Hironaka's reduction of singularities to
$M_0$, we may assume that all the points in $M_0$ are nonsingular.
Also by classical results on reduction of singularities, we obtain
the following statement:
\begin{lemma}
There is a birational morphism $M\rightarrow M_0$ such that the
center $Y$ of $\nu$ at $M$ has dimension equal to $\dim\nu$.
\end{lemma}
\begin{proof} See for instance Vaqui\'{e}'s paper \cite{Vaq}.
\end{proof}
Thus we may assume that $M_0$ is non-singular and the center $Y_0$
of $\nu$ at $M_0$ has dimension equal to $\dim\nu$. If $\dim\nu=2$,
then $Y_0$ is a hypersurface and a generic point of $Y_0$ is always
nonsingular for $\mathcal L$, since the singular locus of $\mathcal
L$ has codimension  at least $2$ in any nonsingular ambient space.

Consider the case  $\dim\nu=\dim Y_0=1$. We blow-up $M_0$ with
center $Y_0$ to get $M_1\rightarrow M_0$. The new center $Y_1$ of
$\nu$ at $M_1$ is a curve that applies surjectively over $Y_0$. We
repeat the procedure to get an infinite sequence
$$
M_0\leftarrow M_1\leftarrow M_2\leftarrow \cdots
$$
where the center $Y_i$ of $\nu$ at $M_i$ is a curve that applies
surjectively over $Y_{i-1}$. In this situation we can apply the
equireduction arguments in \cite{Can-M-R}, (see also \cite{San}) to
obtain an elementary $\mathcal L$ at a generic point of $Y_i$ for
$i>>0$. These arguments are actually of two-dimensional nature and
the invoked equireduction results are very similar to the original
Seidenberg's result in \cite{Sei}.
\part{Globalization}
In this Part 3 we prove the global result stated in Theorem
\ref{teodos}. To do this we will apply the axiomatic version of the
Zariski's Patching of Local Uniformizations \cite{Zar} that has been
developed by O. Piltant in \cite{Pil}.

Let us state the axiomatic version of the patching of local
uniformizations that we need to use. Fix a field of rational
functions $K/k$ of transcendence degree three over $k$. We take $k$
an algebraically closed field of characteristic zero, even if
Piltant's result is more general than that. Assume that we have an
assignation
$$
M\mapsto \mbox{RegP}(M)\subset M
$$
that chooses a nonempty Zariski open subset $\mbox{RegP}(M)\subset
M$  for each projective model $M$ of $K$. This map can be thought of
by saying that $\mbox{RegP}(M)$ is the {\em set of points of $M$
that satisfy the property ``P''}. Let us introduce now a list of
axioms for globalization.
\begin{quote}
{\em Axiom I}. For each projective model $M$ of $K$ the set
$\mbox{RegP}(M)$ is a nonempty Zariski open set  contained in the
set of regular points $\mbox{Reg}(M)$ of $M$. Moreover the
definition of $\mbox{RegP}(M)$ is local in the sense that given two
projective models  $M$ and $M'$,  two Zariski open sets $U\subset M$
and $U'\subset M'$ and an isomorphism $\phi:U\rightarrow U'$, then
$\phi(\mbox{RegP}(M)\cap U)=\mbox{RegP}(M')\cap U'$.
\end{quote}
The next axiom says that $\mbox{RegP}(M)$ has a good behavior under
blow-up.
\begin{quote}
{\em Axiom II}. Let $Y\subset M$ be an irreducible algebraic
subvariety of $M$ such that $Y\cap\mbox{RegP}(M)\ne \emptyset$. Let
$\pi: M'\rightarrow M$ be the blow-up with center $Y$. There is a
nonempty Zariski open subset $V_Y$ of $Y\cap \mbox{RegP}(M)$ defined
by the property that $\pi^{-1}(V_Y)\subset \mbox{RegP}(M')$.
\end{quote}
In what follows we take $V_Y$ to be the largest possible between the
subsets $V\subset Y\cap \mbox{RegP}(M)$ such that
$\pi^{-1}(V)\subset \mbox{RegP}(M')$. This determines $V_Y$
uniquely. As a consequence of Axiom II, if we blow-up a point $P\in
\mbox{RegP}(M)$ then we have $\pi^{-1}(P)\subset \mbox{RegP}(M')$.
The open set $V_Y$ is called {\em set of permissibility for $Y$}. We
say that $Y$ is {\em permissible if } $V_Y=Y$. We need also the
notion of {\em strong permissibility\/}. If $Y\subset M$ is a point
or a hypersurface that cuts $\mbox{RegP}(M)$, we define the open set
of strong permissibility $W_Y$ as $W_Y=V_Y$. Assume that $Y$ is an
irreducible curve and let $P\in V_Y$. We say that $Y$ is {\em
strongly permissible at $P$} iff the following property holds
\begin{quote}
Let $ M=M_0\leftarrow M_1\leftarrow M_2\leftarrow\cdots\leftarrow
M_N$ be a finite sequence of blow-ups centered at points $P_i\in
M_i$, such that $P_0=P$ and $P_i$ projects over $P_{i-1}$ and is in
the strict transform $Y_i$ of $Y$. Then $Y_N$ is permissible at
$P_N$ (that is $P_N\in V_{Y_N}$).
\end{quote}
We denote by $W_Y\subset V_Y$ the set of points where $Y$ is
strongly permissible and we say that $Y$ is {\em strongly
permissible }iff $W_Y=Y$.
\begin{quote}
{\em Axiom III}. Let $Y$ be a curve in $M$ such that
$Y\cap\mbox{RegP}(M)\ne \emptyset$. There is a finite sequence $$
M=M_0\leftarrow M_1\leftarrow M_2\leftarrow\cdots\leftarrow M_N=M'$$
of blow-ups with center in closed points such that the strict
transform $Y'$ of $Y$ is strongly permissible.
\end{quote}
In fact, the centers in Axiom III can be chosen in $Y-W_Y$ at each
step; this also shows that $W_Y$ is a nonempty open set of $Y$.

We also need another axiom (of principalization), that can be seen
as a result on {\em conditionated desingularization}
\begin{quote}
{\em Axiom IV [Principalization]}. Given a (normal) projective model
$M_0$ of $K$ and an ideal sheaf ${\mathcal I}\subset{\mathcal
O}_{M_0}$, there is a projective birational morphism
$\pi:M\rightarrow M_0$ such that
\begin{enumerate}
\item ${\mathcal I}{\mathcal O}_{M}$ is
locally principal in $\pi^{-1}(\mbox{RegP}(M_0))$.
\item ${\pi^{-1}(\mbox{RegP}(M_0))}\subset\mbox{RegP}(M)$.
\item The induced map
$\pi^{-1}(\mbox{RegP}(M_0)\cap U )\rightarrow \mbox{RegP}(M_0)\cap
U$ is an isomorphism where $U$ is the open set of the points $p$ of
$M_0$ such that ${\mathcal I}_p$ is principal.
\end{enumerate}
\end{quote}

The last axiom states the existence of Local Uniformization.
\begin{quote}
{\em Axiom V [Local Uniformization]}. Let $\nu$ be a $k$-valuation
of $K$. There is a projective model $M$ of $K$ such that the center
$Y$ of $\nu$ in $M$ cuts $\mbox{RegP}(M)$, that is  $Y\cap
\mbox{RegP}(M)\ne \emptyset$.
\end{quote}
With this axioms, it is possible to reproduce Zariski's arguments in
\cite{Zar2} for the patching of local uniformizations and we can
state the following result
\begin{theorem}[Piltant]
\label{teo:piltant}
 Assume that the assignation $M\mapsto \mbox{RegP}(M)$
satisfies to the axioms I,II,III, IV and V above. Consider a
projective model $M_0$ of $K$. Then there is a birational projective
morphism $M\rightarrow M_0$ such that $\mbox{RegP}(M)=M$.
\end{theorem}
This statement is slightly more restrictive that the result proved
by Piltant in \cite{Pil}. It is the result we need to get our global
statement for the case of foliations. Now, in order to prove Theorem
\ref{teodos}, we just need to prove the following statement
\begin{proposition}
\label{pro:axiomas}
 Let us consider a foliation
${\mathcal L}\subset\mbox{Der}_kK$. The assignation
$$
M\mapsto \mbox{RegLog}_{\mathcal L}(M)=\{P\in M; P \mbox{ is
log-elementary }\}
$$
satisfies to the axioms I,II,III,IV and V.
\end{proposition}
The Local Uniformization Axiom is given by Theorem \ref{teouno}. The
first axiom is evident from the local definition of log-elementary
points, let us just point that $\mbox{RegLog}_{\mathcal L}(M)$ is
non-empty since the non singular points of $\mathcal L$ in
$\mbox{Reg}(M)$ are in the complement of a closed subset of
codimension bigger or equal than two.

The axioms II and III come from the general computations done in
\cite{Can} concerning the definition and properties of permissible
centers in terms of the adapted multiplicity. More precisely, in
theorem 3.1.4. of \cite{Can} is proved the stability of the adapted
order under blow-up (the log-elementary singularities are defined to
have adapted order less or equal to one). The permissibilyzing and
permissibility properties come from the results on stationary
sequences in the section 3.3. of \cite{Can}.

Finally axiom IV of principalization has been explicitly proved for
the case $\mbox{RegP}(M)=\mbox{RegLog}_{\mathcal L}(M)$ by Piltant
in \cite{Pil}, Proposition 4.2.

Now, Theorem \ref{teodos} is a consequence of Proposition
\ref{pro:axiomas} and Theorem \ref{teo:piltant}.


\end{document}